\documentclass[11pt]{amsart}
\usepackage[top=1.2in,bottom=1.2in,left=1.2in,right=1.2in,marginpar=1in]{geometry}
\usepackage{subcaption}
\usepackage{graphicx}
\usepackage{amssymb}
\usepackage{epstopdf}
\usepackage{subcaption}
\usepackage[mathcal]{eucal}
\usepackage{pinlabel}
\usepackage{hyperref}
\usepackage{enumitem}
\usepackage[all,cmtip]{xy}
\usepackage[textsize=tiny]{todonotes}

\usepackage{scalerel}

\makeatletter
\newcommand*{\bigcdot}{}
\DeclareRobustCommand*{\bigcdot}{%
  \mathbin{\mathpalette\bigcdot@{}}%
}
\newcommand*{\bigcdot@scalefactor}{.75}
\newcommand*{\bigcdot@widthfactor}{1.15}
\newcommand*{\bigcdot@}[2]{%
  \sbox0{$#1\vcenter{}$}
  \sbox2{$#1\cdot\m@th$}%
  \hbox to \bigcdot@widthfactor\wd2{%
    \hfil
    \raise\ht0\hbox{%
      \scalebox{\bigcdot@scalefactor}{%
        \lower\ht0\hbox{$#1\bullet\m@th$}%
      }%
    }%
    \hfil
  }%
}
\makeatother
\newcommand{\nc}{\newcommand}
\nc{\dmo}{\DeclareMathOperator}
\nc{\nt}{\newtheorem}

\nt{theorem}{Theorem}[section]
\nt{problem}{Problem}{\Alph{problem}}
\nt{solution}{Solution}{\Alph{solution}}
\nt{proposition}[theorem]{Proposition}
\nt{corollary}[theorem]{Corollary}
\nt{lemma}[theorem]{Lemma}
\nt{criterion}[theorem]{Criterion}
\nt*{theorem*}{Theorem}

\nc{\M}{\mathcal{M}}
\nc{\C}{\mathcal{C}}

\nc{\cut}{\!\ssearrow\!}

\dmo{\Mod}{Mod}
\dmo{\Teich}{Teich}
\dmo{\PMod}{PMod}
\dmo{\LMod}{LMod}
\dmo{\SMod}{SMod}
\dmo{\I}{\mathcal{I}}
\dmo{\SL}{SL}
\dmo{\PSp}{PSp}
\dmo{\PSL}{PSL}
\dmo{\Homeo}{Homeo}
\dmo{\Aut}{Aut}
\dmo{\Sq}{Sq}
\dmo{\Rot}{Rot}
\dmo{\Cub}{Cub}

\nc{\Z}{\mathbb Z}
\nc{\Q}{\mathbb Q}
\nc{\N}{\mathcal N}
\nc{\R}{\mathbb R}
\nc{\F}{\mathcal F}
\nc{\A}{\mathcal A}
\nc{\X}{\mathcal X}
\nc{\Y}{\mathcal Y}
\nc{\T}{\mathcal T}
\nc{\dig}{s}
\nc{\degr}{d}

\nc{\p}[1]{\bigskip\noindent\emph{#1.}}

\usepackage{ulem}
\title{Twisting cubic rabbits}
\author{Justin Lanier}
\author{Rebecca R. Winarski}

\address{Justin Lanier \\ Department of Mathematics\\ University of Chicago \\ 5734 S. University Avenue \\ Chicago, IL 60637}
\email{jlanier8@uchicago.edu}

\address{Rebecca R. Winarski \\  Department of Mathematics and Computer Science\\ College of the Holy Cross\\
1 College Street \\
Worcester, MA 01610}
\email{rebecca.winarski@gmail.com}

\begin{document}
\maketitle

\begin{abstract}
We solve an infinite family of twisted polynomial problems that are cubic generalizations of Hubbard's twisted rabbit problem. We show how the result of twisting by a power of a certain Dehn twist depends on the 9-adic expansion of the power. For the cubic rabbit with three post-critical points, we also give an algorithmic solution to the twisting problem for the full pure mapping class group. 
\end{abstract}

\section{Introduction}

In this paper we give closed-form and algorithmic solutions to cubic generalizations of the twisted rabbit problem. The twisted rabbit problem was formulated by Hubbard and asks: to which polynomial is $D_x^m \circ r$ Thurston equivalent? Here $r$ is the post-critically finite quadratic polynomial known as the Douady rabbit with post-critical set $P_r$ consisting of three points, $D_x$ is the left-handed Dehn twist on the simple closed curve encircling the ``ears" of the rabbit, and $m$ varies over the integers. Bartholdi--Nekrashevych gave a closed-form solution to this problem, and they also gave an algorithm for determining the Thurston equivalence class of $g \circ r $ for arbitrary pure mapping classes $g \in \PMod(\R^2,P_r)$ \cite{BaNe}. Their method of analyzing self-similar groups associated to post-critically finite polynomials can be used to solve other twisted polynomial problems; in the same paper they carry this out for other quadratic polynomials where the post-critical set consists of three points.

In work with Belk and Margalit, we gave an alternative solution to the twisted rabbit problem and some additional twisted polynomial problems. To do so, we followed the strategy of  Bartholdi--Nekrashevych but applied methods of combinatorial topology instead of their more group theoretic approach; see \cite{BLMW} for a comparison of the methods. The present paper is a close sequel to Section 5 of that paper. There we gave a solution to the original twisted rabbit problem as well as an infinite family of related problems: for each $n$, there is an analogue of the rabbit polynomial that is a quadratic polynomial with post-critical set of size $n$. For each of these polynomials we defined an analogous Dehn twist cyclic subgroup and gave a closed-form solution to the associated family of twisting problems using a uniform argument.

Increasing the length of the critical cycle is one way to generalize the rabbit polynomial. Another is to change the degree of the polynomial. In this paper we consider twists of a unicritical cubic polynomial with post-critical set of size 3 called the cubic rabbit, as well as an infinite family of unicritical cubic polynomials---``many-eared" cubic rabbits---that have  critical cycles of arbitrary length. In each case we give a closed-form solutions to twisting problems on a cyclic subgroup of Dehn twists. The resulting maps depend on the 9-adic expansion of the power of the twist. For the cubic rabbit with 3 post-critical points, we also give two algorithmic solutions to the problem of determining the Thurston equivalence class of post-composing the cubic rabbit with arbitrary pure mapping classes. One algorithm uses the wreath recursion approach employed by Bartholdi--Nekrashevych, while the other applies an elementary word length argument.

\begin{figure}
    \centering
    \includegraphics[width=.35\textwidth]{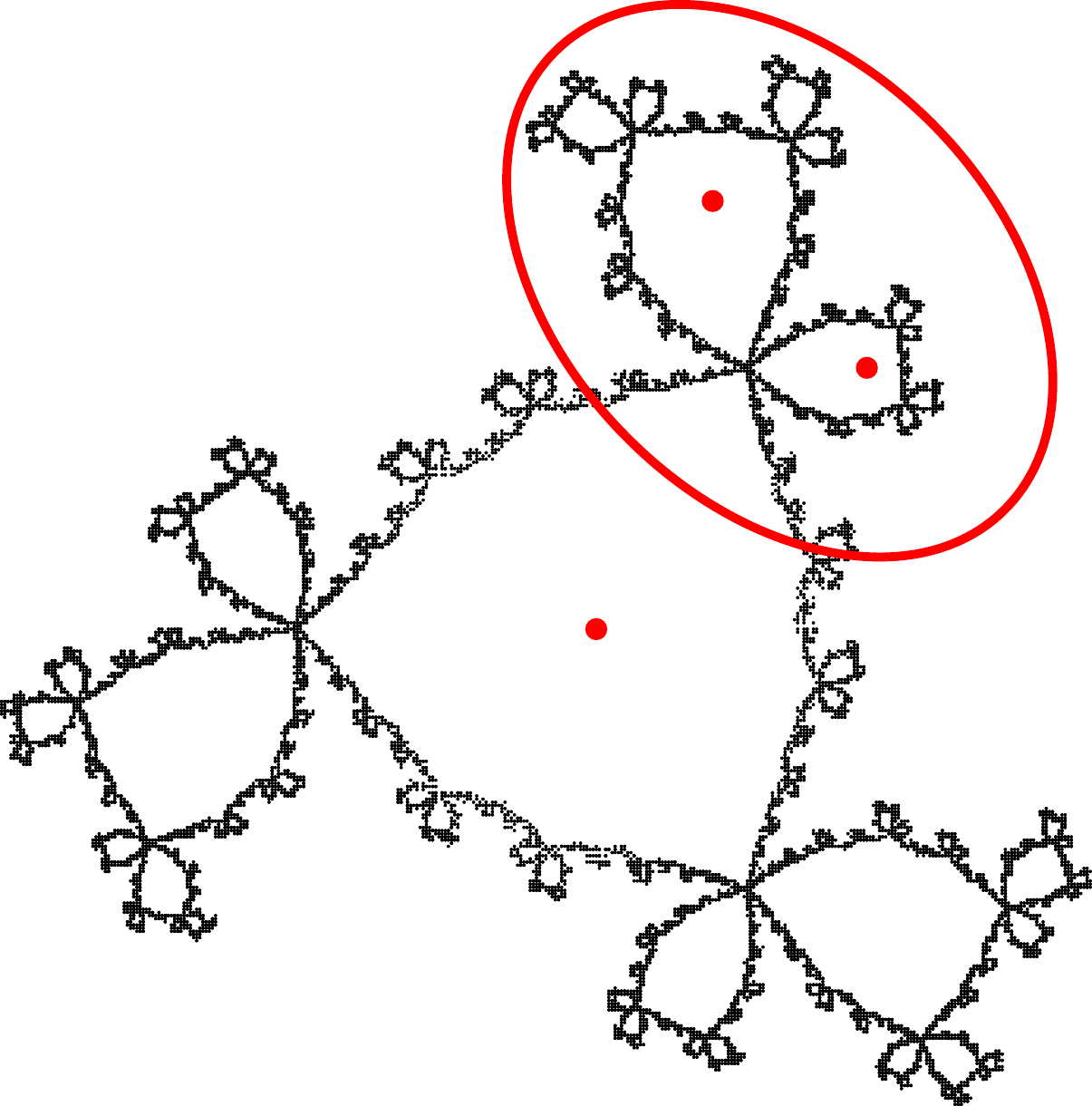}\qquad \qquad
    \includegraphics[width=.35\textwidth]{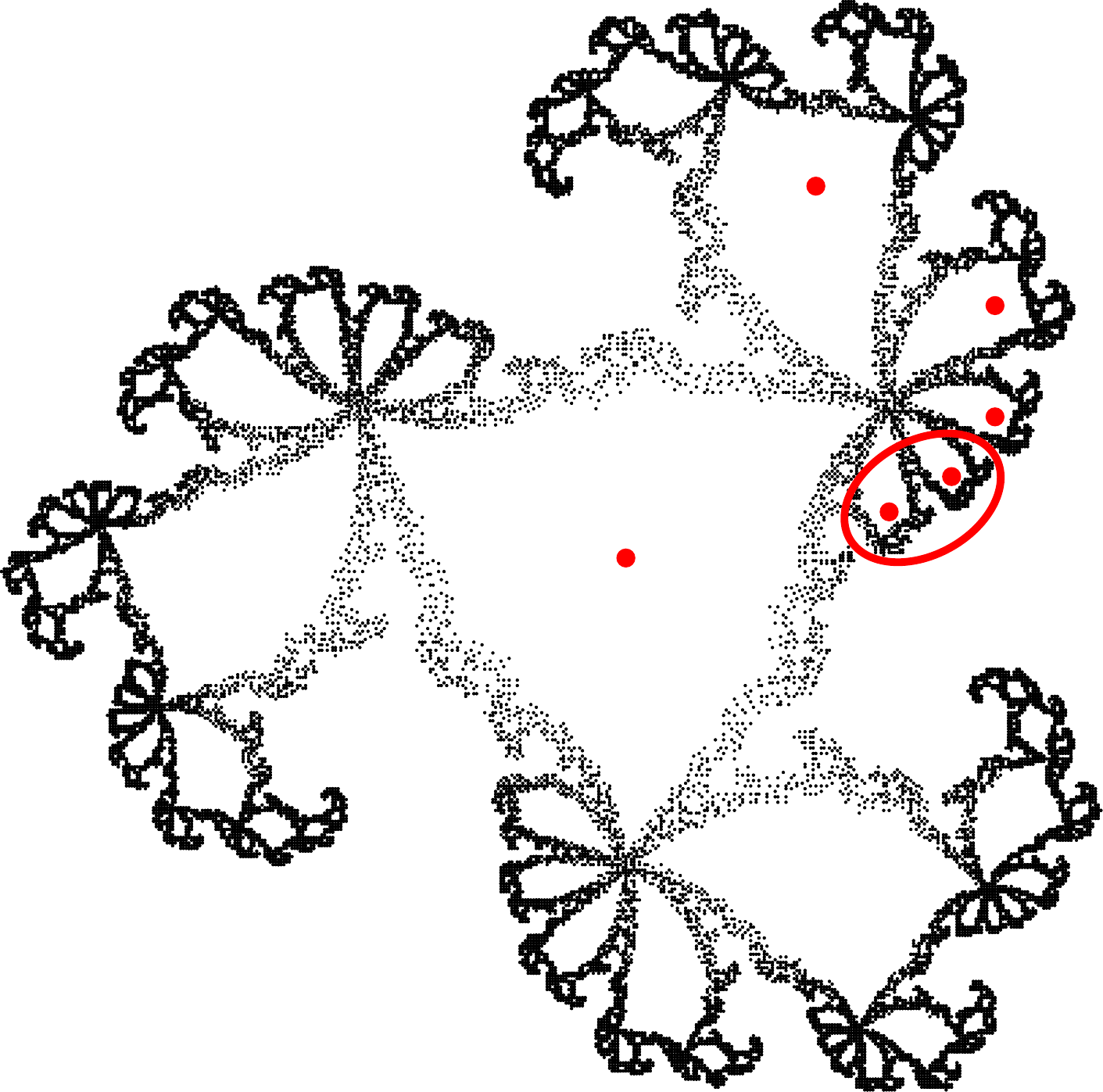}
    \caption{For $n=3$ and $6$: the Julia set for $R_n$, the post-critcal set for $R_n$, and the curve $x_n$.}
    \label{fig:R3R6}
    
\end{figure} 

In Figure~\ref{fig:R3R6} we depict the Julia sets and post-critical sets of cubic rabbits $R_3$ and $R_6$ with 3 and 6 post-critical points, respectively. These two polynomials are $R_n(z)=z^3+c_n$ with $c_3 \approx 0.558+0.540i$ and $c_6 \approx 0.510 + 0.089i$. They belong to an infinite family $R_n$ of polynomials that share combinatorial and dynamical properties. For each map, Figure~\ref{fig:R3R6} also shows a curve $x_n$ that encircles the two points $R_n(0)$ and $R_n^2(0)$.

Up to Thurston equivalence, there are four unicritical cubic polynomials with 3 post-critical points: the cubic rabbit, the cubic airplane, and their complex conjugates. We notate these as $R_3,\overline{R}_3,A_3$, and $\overline{A}_3$; see Section~\ref{sec:cubic} for their descriptions. All four of these arise as twists of the cubic rabbit $R_3$ about the curve $x_3$.

\begin{theorem}
\label{thm:cubicrabbit}
Let $x=x_3$ and let $m \in \Z$. Let $\dig$ be the right-most digit in the 9-adic expansion of $m$ that is not a 0, 4, or 8, if this exists. Then

\[
D_x^m R_3 \simeq
\begin{cases}
R_3&\text{ if the 9-adic expansion of  }m\geq0\text{ contains only 0's, 4's, and 8's}\\
\overline{R}_3&\text{ if the 9-adic expansion of  }m<0\text{ contains only 0's, 4's, and 8's}\\
A_3&\text{ if }\dig= \text{1, 5, or 6}\\
\overline{A}_3&\text{ if }\dig= \text{2, 3, or 7}\\
\end{cases}
\]

\end{theorem}

\noindent As illustrations of this result, the integers $89$ and $-77$ have $9$-adic expansions $108_9$ and $\dots88804_9$. (Note that $\dots 888_9=-1$.) Therefore $D_x^{89} R_3 \simeq A_3$ and $D_x^{-77} R_3 \simeq \overline{R}_3$.

As $n$ increases, the number of unicritical cubic polynomials with $n$ post-critical points increases exponentially. As an organizing principle, we can find families of polynomials, indexed by $n$, with shared combinatorics and dynamics. The maps $D_x^mR_n$ in the ``many-eared" twisted cubic rabbit problem are Thurston equivalent to polynomials in nine such families for all $n\geq 4$. We describe the families that appear in Theorem~\ref{thm:many} in Section~\ref{sec:cubicn}.

\begin{theorem}
\label{thm:many}
Let $n \geq 4$, let $x=x_n$, and let $m \in \Z$.  If $m \neq 0$, let $\dig$ be the right-most non-zero digit of the 9-adic expansion of $m$.  Then $D_x^m R_n$ is equivalent to the map:

\begin{center}
\begin{table}[!ht]
\begin{tabular}{||r |c| c |c| c| c| c| c| c|c ||} 
 \hline
  & $m=0$ & $\dig=1$ & $\dig=2$ & $\dig=3$ &  $\dig=4$ &  $\dig=5$ & $\dig=6$ &  $\dig=7$ & $\dig=8$\\
 \hline
\rule{0pt}{18pt}$D_x^m R_n\simeq $ & $R_n$ & $A_n$ & $\overline{A}_n$ & $K_{n,1}$ & $B_n$ & $\overline{Y}_n$ & $K_{n,2}$ & $Y_n$ & $\overline{B}_n$\\
 \hline
\end{tabular}
\label{tab:Thm2}
\end{table}
\end{center}

\end{theorem}

To prove Theorems~\ref{thm:cubicrabbit} and \ref{thm:many}, we follow the strategy of Bartholdi--Nekrashevych \cite{BaNe}, which consists of two main steps: producing reduction formulas and calculating base cases. Producing the reduction formulas involves lifting mapping classes through $R_n$. To calculate the base cases, we use the approach of Belk--Lanier--Margalit--Winarski \cite{BLMW} to find the topological Hubbard trees and accompanying data for a small number of twisted polynomials, to which all other cases reduce.

The $9$-adic expansion of $m$ appears in the Theorems~\ref{thm:cubicrabbit} and \ref{thm:many} for the same reason that the \mbox{$4$-adic} expansion of $m$ appears in the solution of the original twisted rabbit problem: for a \mbox{degree-$d$} rabbit $R$ with $n$ post-critical points and a similarly situated curve $x$, twisting $R$ by the mapping class $D_x^{d^2k}$ always yields the same equivalence class of topological polynomial as twisting by $D_x^{k}$. Thus powers of $d^2$ can be divided out of $m$ without affecting the equivalence class of $D_x^{m}R$, and this is the same as dropping initial 0's in the $d^2$-adic expansion $m$. It is possible for similar reductions to happen for other residue classes $s$ mod $d^2$, whenever $D_x^{d^2k+s}R$ happens to be equivalent to $D_x^{k}R$, as shown through a lifting calculation. For these values of $s$, initial $s$'s in the $d^2$-adic expansion of $m$ can also be dropped. For both $d=2$ and $3$, additional reductions occur when $n=3$ (the digits $4$ and $8$ in Theorem~\ref{thm:cubicrabbit}) but do not occur for $n \geq 4$. Forthcoming work of Mukundan and the second author further explores this phenomenon by giving solutions to twisted degree-$d$ rabbit problems for $n=3$ and all $d\geq 2$.

\p{Outline of the paper} In Section~\ref{sec:background} we review the relevant definitions and background. This includes techniques we developed with Belk and Margalit; for full details, see \cite{BLMW}. In Section~\ref{sec:cubic} we define the cubic rabbit polynomial and prove Theorem~\ref{thm:cubicrabbit}. In Section~\ref{sec:algo} we give two algorithmic solutions to the problem of determining the Thurston equivalence class of post-composing the cubic rabbit with arbitrary pure mapping classes.  In Section 5 we define the ``many-eared" cubic rabbits and prove Theorem~\ref{thm:many}. 

\p{Acknowledgments} This work began while the authors were collaborating with Jim Belk and Dan Margalit on the paper \cite{BLMW}.  The first author was supported by the National Science
Foundation under Grant No. DGE-1650044 and Grant No. DMS-2002187. The second author was supported by the National Science Foundation under Grant No.\ DMS-2002951. This material is based upon work supported by the National Science Foundation
under Grant No. DMS-1928930 while the second author participated in a program hosted
by the Mathematical Sciences Research Institute in Berkeley, California, during Spring 2022.

\section{Background and techniques}\label{sec:background}
We refer the reader to joint work of the authors with Belk and Margalit  for additional background \cite{BLMW}, but we review here the definitions and results essential to stating and solving our twisted cubic rabbit problems.

\p{Topological polynomials} A {\it topological polynomial} $f$ is an orientation-preserving branched cover $f:\R^2\rightarrow \R^2$ with finite degree $\degr$ greater than 1 and finitely many critical points.  The {\it post-critical set} $P_f$ is the forward orbit of the critical points of $f$.  We say that a topological polynomial $f$ is {\it post-critically finite} if $P_f$ is finite.

Let $f:(\R^2,P_f)\rightarrow (\R^2,P_f)$ and $g:(\R^2,P_g)\rightarrow(\R^2,P_g)$ be topological polynomials with $|P_f|=|P_g|$.  We say that $f$ and $g$ are {\it equivalent} (or {\it Thurston equivalent}, or {\it combinatorially equivalent}) if there exist orientation-preserving homeomorphisms $h_1,h_2:(\R^2,P_f)\rightarrow  (\R^2,P_g)$ such that $h_1f=gh_2$ and $h_1$ and $h_2$ are isotopic relative to $P_f$. When topological polynomials $f$ and $g$ are equivalent, we write $f \simeq g$.

\p{Hubbard trees} Every post-cricially finite polynomial has an associated tree called its {\it Hubbard tree}.  Let $f$ be a polynomial with finite post-critical set $P_f$.  Following Douady--Hubbard \cite{DH1, DH2}, we define the Hubbard tree for $f$ to be the union of the regulated (or allowable) arcs of $P_f$ in the filled Julia set for $f$.  The Hubbard tree $H_f$ of $f$ is invariant under $f$, that is, $f(H_f)\subseteq H_f$.  Moreover, if $f$ is a polynomial with Hubbard tree $H_f$ and $g$ is a topological polynomial that is equivalent to $f$, then there is an isotopy class of tree associated to $g$, which we obtain by pulling back $H_f$ through the equivalence.  We call this tree the {\it topological Hubbard tree} $H_g$ of the topological polynomial $g$.

\p{Lifting trees}  Let $f$ be a post-critically finite topological polynomial with post-critical set $P_f$.  Let $T\subset \R^2$ be a tree containing $P_f$ such that all edges of the tree are contained in a path between points in $P_f$.  In particular, all leaves of $T$ must lie in $P_f$.  The {\it lift} of $T$ under $f$ is defined as a composition of two operations and is notated as $\lambda_f(T)$. First, take the preimage $f^{-1}(T)$, which is also a tree in $\R^2$. Then take the hull relative to $P_f$: that is, remove any edges of $f^{-1}(T)$ that are not part of paths between points in $P_f$. The composition of these two operations produces a tree $\lambda_f(T)$ in $\R^2$ containing $P_f$. The topological Hubbard tree $H_f$ for $f$ is isotopic (relative to $P_g$) to its lift through $f$ and so is an invariant tree for $f$. However, it need not be the unique tree with this property.

\p{Tree maps and angle assignments} While the topological Hubbard tree alone is not in general sufficient to determine the equivalence class of a topological polynomial $f$, we can endow it with additional information that then determines the equivalence class completely. The first piece of data needed is the the restriction of $f$ to the edges of the tree. This data for instance distinguishes the rabbit and corabbit polynomials in the quadratic case and in their higher-degree generalizations. The second piece of data is an invariant angle assignment; see Poirier \cite{poirier} or Belk--Lanier--Margalit--Winarski \cite[Section 3]{BLMW} for details. This data, for instance, distinguishes what we call the airplane and the coairplane cubic polynomials. 

\p{Lifting mapping classes} We follow the strategy of Bartholdi--Nekrashevych of replacing a mapping class with its lift under a branched cover in order to determine the equivalence class of a twisted topological polynomial.

Let $f:(\R^2,P_f)\rightarrow(\R^2,P_f)$ be a branched cover and $h:(\R^2,P_f)\rightarrow(\R^2,P_f)$ be a homeomorphism.  We say that $h$ lifts under $f$ if there exists a homeomorphism $\widetilde{h}:(\R^2,P_f)\rightarrow(\R^2,P_f)$ such that $hf=f\widetilde{h}$.  In this case we say that $h$ is {\it liftable} and that $h$ {\it lifts to} $\widetilde{h}$.  Because homotopy is preserved under lifting, the homotopy classes of liftable homeomorphisms form a (finite index) subgroup of the mapping class group called the  {\it liftable mapping class group} $\LMod(\R^2,P_f)$.

If $h$ lifts to $\widetilde{h}$ under $f$, then $hf$ and $\widetilde{h}f$ are equivalent.  More generally, for any $h \in \PMod(\R^2,P_f)$, there exists $g$ such that $g^{-1}h \in \LMod(\R^2,P_f)$. Then $g^{-1}h$ is liftable and let $\psi(g^{-1}h)$ denote the lift. In this case, $hf$ is equivalent to $\psi(g^{-1}h)gf$ (cf. \cite[Proposition 4.1]{BaNe} and \cite[Lemma 5.1]{BLMW}).  When $f$ is specified, we use the notation $$h\stackrel{g}{\leadsto}\psi(g^{-1}h)g$$ to indicate equivalence obtained by choosing $g$ as a coset representative of $h \in \LMod(\R^2,P_f)$ and lifting.  Observe that when $h\in\LMod(\R^2,P_f)$, we may choose $g$ to be the identity and we suppress $g$ in the notation.

For $f$ a topological polynomial and $g$ and $h$ mapping classes, whenever $gf \simeq hf$ we say $g \sim h$, where it is understood that this is with respect to the map $f$.

\p{A criterion for trivial lifts} To find the result of lifting a curve, one can simply take the preimage and forget the inessential components. However, there exists a useful criterion that guarantees that the preimage of a curve will have only inessential components. To this end, we prove a generalization of Lemma~5.2 in Belk--Lanier--Margalit--Winarski \cite{BLMW}, extending it to topological polynomials of degree higher than 2. We first require some definitions.

Let $f$ be a unicritical topological polynomial of degree $\degr$.  A {\it branch cut} $b$ for $f$ is a arc in $(\R^2,P_f)$ connecting the critical value of $f$ to $\infty$.  The preimage $f^{-1}(b)$ is a union of $\degr$ arcs that intersect only at the critical point.  The complement of $f^{-1}(b)$ in $\R^2$ has $\degr$ components.  We say that $b$ is a {\it special branch cut} for $f$ if all points in $P_f$ lie in the same component of $f^{-1}(b)$. A special branch cut for $R_3$ (and $R_n$, similarly) is the straight arc from the critical value to $\infty$ that avoids the interior of the triangle determined by $P_{R_3}$; this is depicted as the arc $b$ in Figure~\ref{fig:D_x(z)}.

Let $c$ be a simple closed curve that is isotopic to the boundary of a disk $D$ that contains exactly two marked points.  A {\it defining arc} for $c$ is a simple arc contained in $D$ that has an endpoint at each marked point in $D$.  Up to isotopy, there is a unique defining arc for any such $c$.

\begin{lemma}\label{lem:triviality2.0}
Let $f$ be a topological polynomial of degree $\degr$ and let $b$ be a special branch cut for $f$.  Suppose $c$ is a curve in $(\R^2,P_f)$ that surrounds exactly two points of $P_f$, neither of which is the critical value, and that $a$ is a defining arc for $c$.  If the algebraic intersection number of $a$ and $b$ is not 0 mod $\degr$, then the lift of $D_c$ is trivial.
\end{lemma}

See the discussion of the reduction formulas in Section~\ref{sec:cubic} for an illustrated application of this lemma. Recall that a curve is trivial if it is homotopic to a point (or a boundary component).  The Dehn twist about a curve $\gamma$ is trivial if and only if $\gamma$ is trivial.

\begin{proof}

Since $c$ does not surround the critical value of $f$, the arc $a$ does not have an endpoint at the critical value.  Therefore $f^{-1}(a)$ consists of $\degr$ arcs that are disjoint, including at their endpoints.  Since the algebraic intersection number of $a$ and $b$ is not 0 mod $\degr$, the endpoints of each component of $f^{-1}(a)$ lie in different connected components of $\R^2 \setminus f^{-1}(b)$.  Since a nontrivial arc must have the property that both of its endpoints are in the same connected component of $\R^2 \setminus f^{-1}(b)$ (the component containing $P_f$), all arcs of $f^{-1}(a)$ are trivial. The lift $\psi(D_c)$ is equal to the product of the Dehn twists about the curves of the boundary of a neighborhood of $f^{-1}(a)$.  Since each such curve surrounds at most one point of $P_f$, this product is trivial.\end{proof}

\begin{figure}[ht]
$\underset{\textstyle\text{(a)}}{\includegraphics[width=45mm]{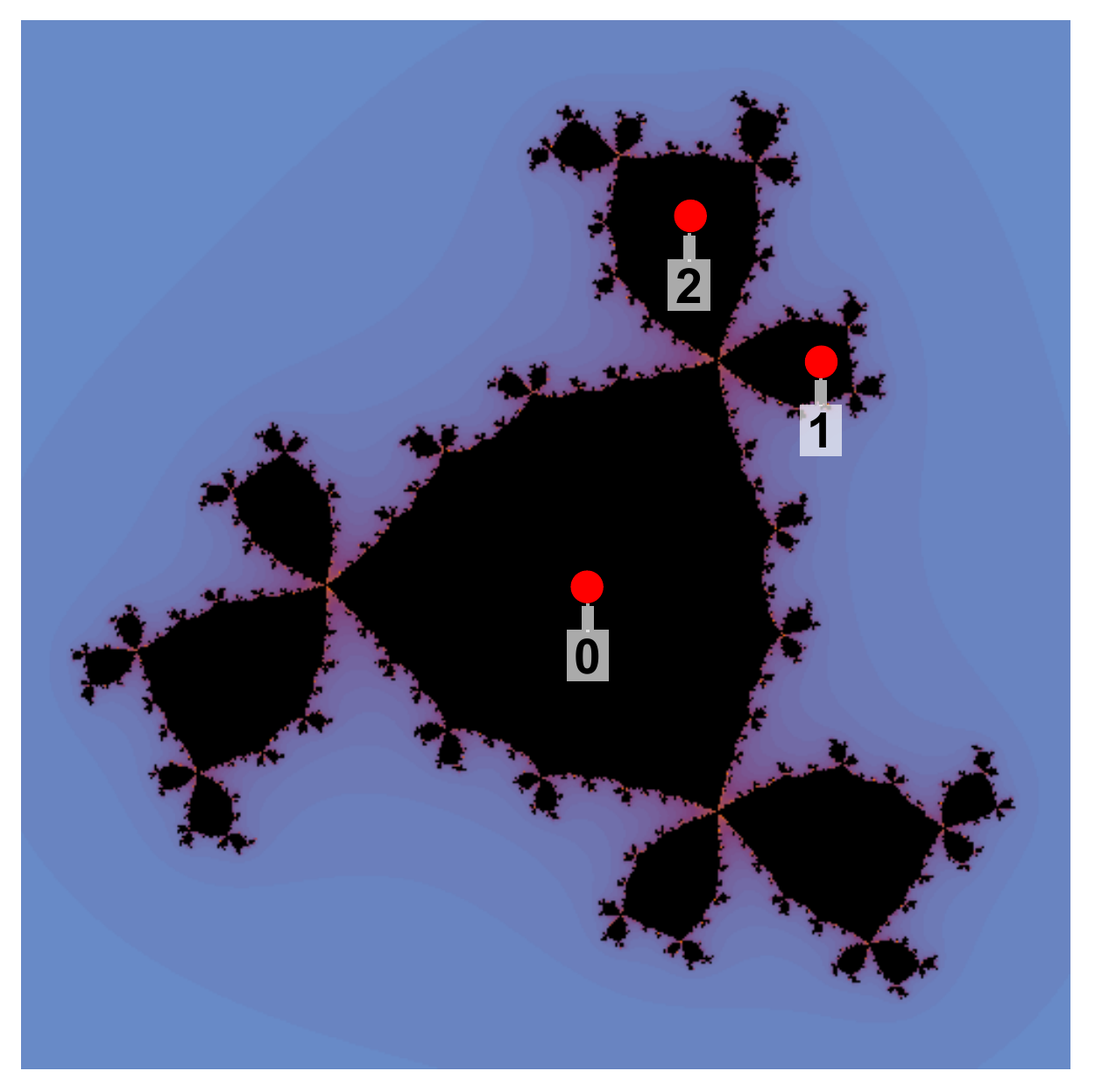}}$
\qquad
$\underset{\textstyle\text{(b)}}{\labellist\small\hair 2.5pt
        \pinlabel $x$ by 0 0 at 95 145
         \pinlabel $y$ by 0 0 at 80 45
        \pinlabel $z$ by 0 0 at 10 90
        \endlabellist\includegraphics[width=25mm]{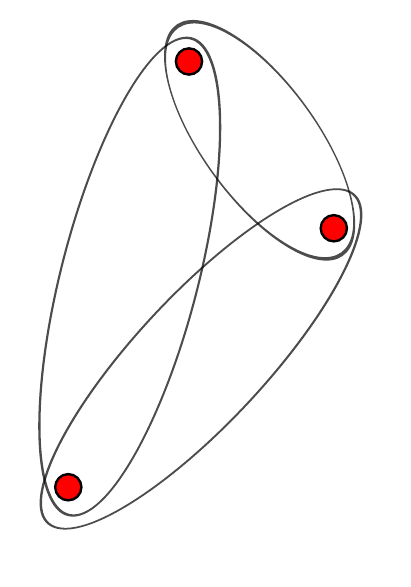}}$
\qquad 
$\underset{\textstyle\text{(c)}}{\includegraphics[width=50mm]{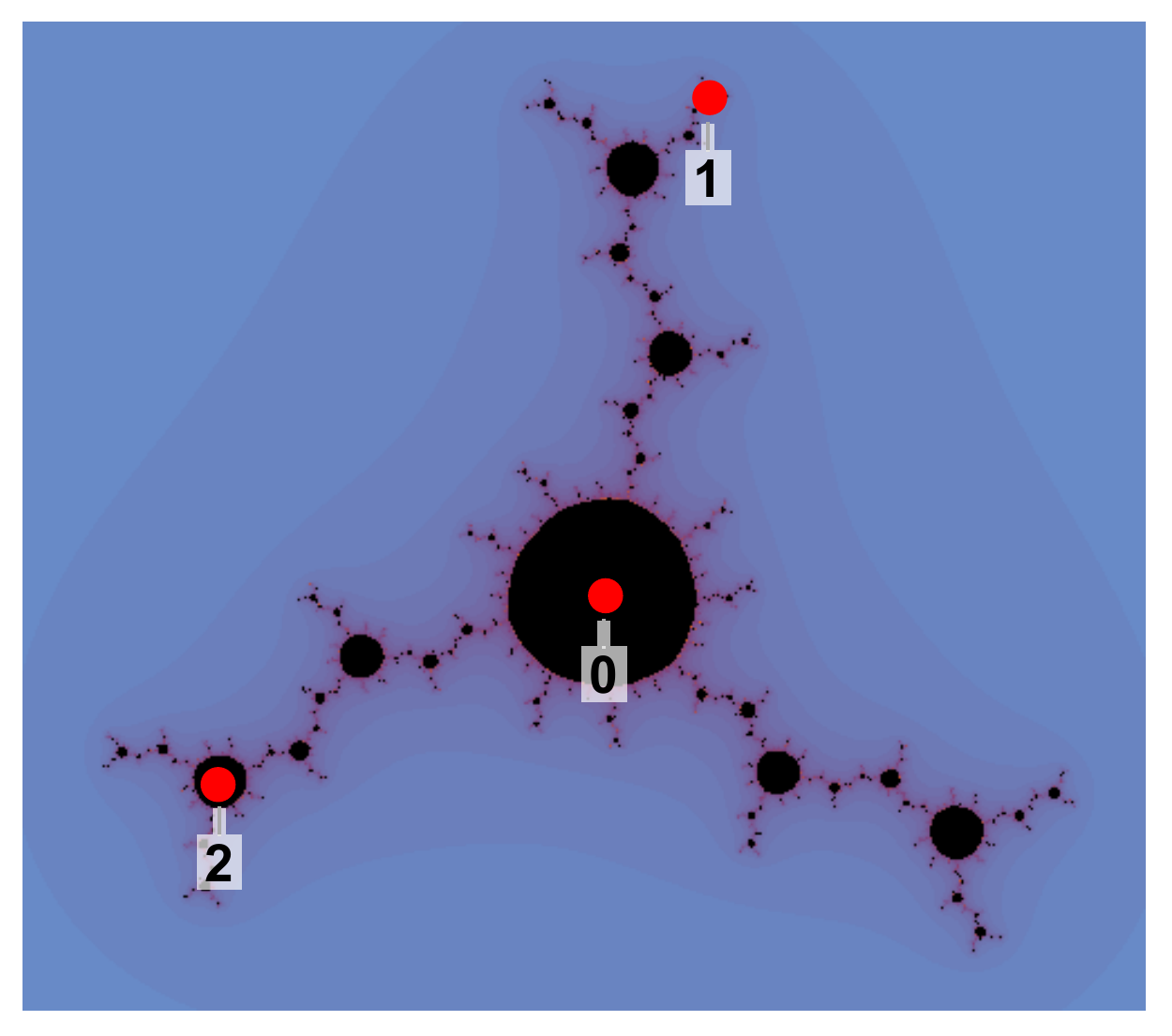}}$       
\caption{(a) The Julia set for the cubic rabbit~$R_3$. (b) The curves $x$, $y$, and $z$ in $(\R^2,P_{R_3})$. (c) The Julia set for the cubic airplane~$A_3$.}
\label{fig:generatorscubic}
\end{figure}

\section{Twisting the cubic rabbit}
\label{sec:cubic}

Every unicritical cubic polynomial is affine conjugate to a polynomial of the form $f(z)=z^3+c$  with critical point $0$. There are eight distinct non-zero solutions to the equation $(c^3+c)^3+c=0$, each of which yields a post-critically finite cubic polynomial with critical portrait $0 \to c \to c^3+c \to 0$. These eight polynomials come in four equivalent pairs; each pair is affine conjugate under the map $z \to -z$. We take as representatives of these four equivalence classes the following four polynomials, which we call the cubic rabbit $R_3$, the cubic corabbit $\overline{R}_3$, the cubic airplane $A_3$, and the cubic coairplane $\overline{A}_3$. These polynomial representatives have the form $z^3+c$, where the four approximate values of $c$ are $0.558+0.540i$, $0.558-0.540i$, $.264+1.260i$, and $.264-1.260i$, respectively. The Julia sets for $R_3$ and $A_3$ are depicted in Figure~\ref{fig:generatorscubic}.

Also depicted in Figure~\ref{fig:generatorscubic} three curves $x$, $y$ and $z$ lying in $(\R^2,P_{R_3})$; they are analogues of the curves in the original twisted rabbit problem. The curve $x$ is obtained as the boundary of a regular neighborhood of the straight line segment between $\Rot_3(0)$ and $\Rot^{2}_3(0)$.

By the Berstein--Levy theorem, a topological polynomial with all post-critical points in a critical cycle is unobstructed \cite{LevyThesis}.  Therefore the composition of a pure mapping class with $R_3$ must be equivalent to one of $R_3$, $\overline{R}_3$, $A_3$, or $\overline{A}_3$. In particular, this is true for all maps $D_x^m R_3$.

\begin{figure}
    \centering
  $\underset{\textstyle\text{(a) $D_x(z)$}}{\labellist\small\hair 2.5pt
        \pinlabel $b$ by 0 0 at 115 40
        \endlabellist   \includegraphics[scale=.6]{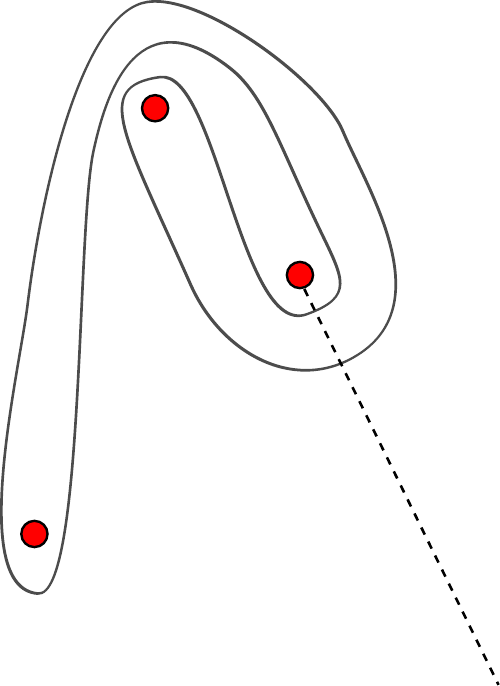}}$\qquad\qquad\qquad
    $\underset{\textstyle\text{(b) The lift of $D_x(z)$}}{\includegraphics[scale=.4]{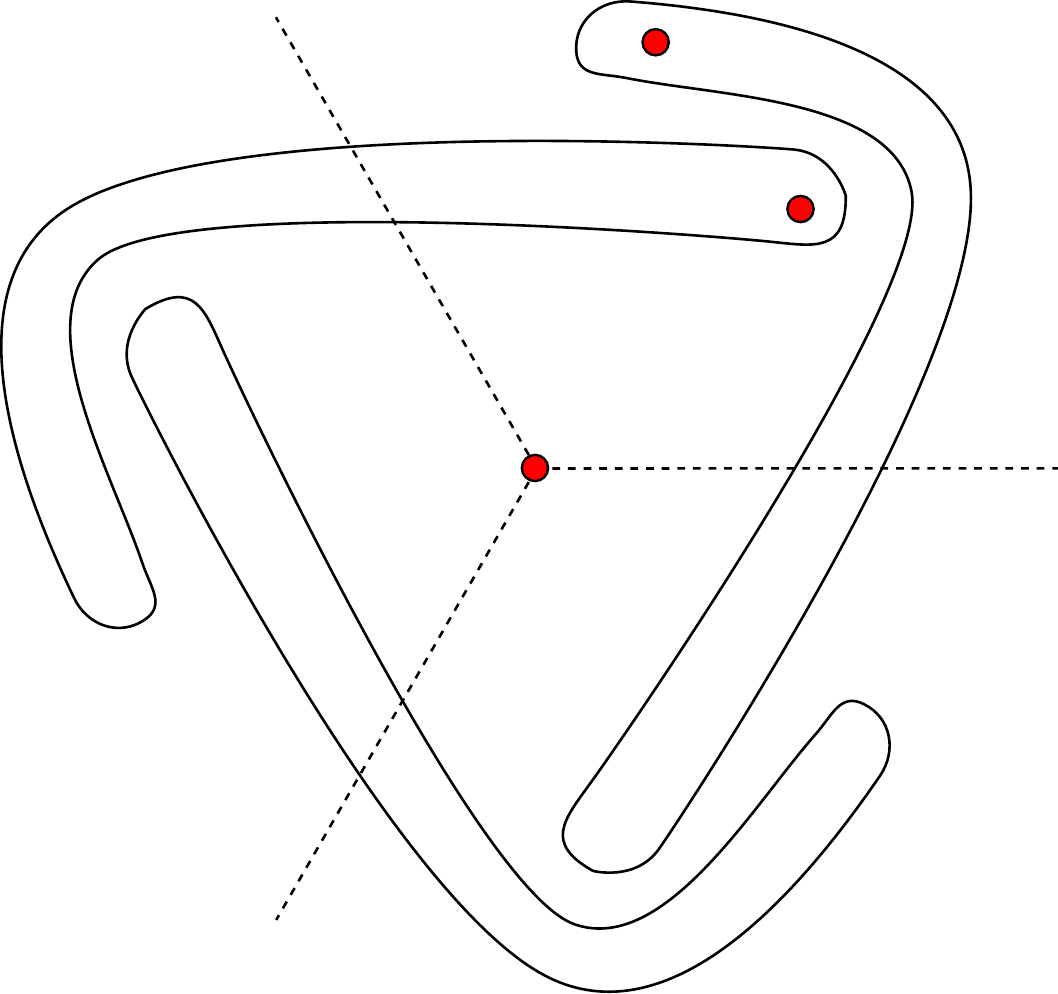}}$
    \caption{In general, the defining arc for $D_x^k(z)$ intersects $b$ in $k$ points.}
    \label{fig:D_x(z)}
\end{figure}

We now describe a topological polynomial that is homotopic (and so also equivalent) to $R_3$ using techniques of combinatorial topology. We describe the map topologically in order to lift isotopy classes of trees and curves through $R_3$ without the need to compute images of analytic maps.  Let $P_{R_3} = \{0,R_3(0),R_3^2(0)\}$ be the post-critical set of $R_3$, and let $\Delta$ be the solid triangle in $\R^2$ with vertex set $P_{R_3}$. Let $\Cub$ be any orientation-preserving triple branched cover $(\R^2,P_{R_3})\to(\R^2,P_{R_3})$ that is branched over 0 and that fixes $\Delta$ pointwise. Any such map fixes any tree contained in $\Delta$.  By the Alexander method, there are two isotopy classes relative to the set $P_{R_3}$ of such branched covers, and these branched covers are equivalent (relative to the set $P_{R_3}$). Let $\Rot_3$ be a homeomorphism of $(\R^2,P_{R_3})$ that rotates the points $P_{R_3}$ counterclockwise and preserves $\Delta$ as a set. Any two such maps are isotopic relative to $P_{R_3}$. Then the map $\Rot_3 \circ \Cub$ is homotopic to the cubic rabbit polynomial relative to $P_{R_3}$; this is straightforward to check using the Alexander method (cf. \cite[Proposition~3.1]{BLMW}).

\subsection{Reduction formulas} We are now ready to explain the first of the two steps in our solution to the twisted cubic rabbit problem. The reduction formulas are:
\[
D_x^{m} R_3 \simeq 
\begin{cases}
D_x^k R_3    & m=9k\\
D_x R_3      & m=9k+1\\
D_x^2 R_3    & m=9k+2\\
D_y R_3      & m=9k+3\\
D_x^k R_3    & m=9k+4\\
D_x R_3      & m=9k+5\\
D_y^2 R_3    & m=9k+6\\
D_x^{-2} R_3 & m=9k+7\\
D_x^k R_3    & m=9k+8\\
\end{cases}.
\]

To verify these formulas, we apply a number of facts about lifting individual curves and their corresponding Dehn twists through $R_3$. For instance, the curve $z$ has three preimages under $R_3$, and the only essential one is homotopic to $x$. Therefore, as in the original twisted rabbit problem, $D_z$ is in $\LMod(\R^2,P_{R_3})$ and we have $D_z \leadsto D_x$.   The preimage of $x$ has a single component, isotopic to $y$.  Therefore $D_x$ is not in $\LMod(\R^2,P_{R_3})$, but $D_x^3$ is in $\LMod(\R^2,P_{R_3})$ and $D_x^3 \leadsto D_y$.  Similarly $D_y^3 \leadsto D_z$. We illustrate this lifting process for several additional curves in Figures~\ref{fig:D_x(z)} and \ref{fig:D_x^-1(y)}.

We now give an example of Lemma~\ref{lem:triviality2.0} in action, showing that $D_x^kD_zD_x^{-k}=D_{D_x^{k}(z)} \leadsto \textrm{id}$ for $k\not\equiv 0\mod 3$ (similarly, $D_{D_y^k(z)} \leadsto \textrm{id}$ for $k\not\equiv 0\mod 3$). Figure~\ref{fig:D_x(z)} illustrates the case $k=1$. Indeed, let $a$ be the the defining arc of $z$.  Then $a$ is disjoint from the special branch cut $b$ and intersects $x$ only once.  Therefore the geometric intersection of $b$ with  $D_x^k(a)$ is $\pm k$ (this is a special case of \cite[Proposition 3.4]{primer} for arcs).  In this situation, the absolute value of the algebraic intersection of $D^k_x(a)$ and $b$ is the same as the geometric intersection because all intersections between $b$ and $D^k_x(a)$ arise from twisting $a$ about $x$ in the same direction. The same argument applies to $D_y^k(z)$.  

We are now prepared to justify the reduction formulas.\\

\noindent \emph{Case 1: $m=9k$.} In this case we have
\[
D_x^{9k}=(D_x^3)^{3k} \ 
\leadsto \ 
D_y^{3k} = (D_y^3)^k \ 
\leadsto \ 
D_z^k \ 
\leadsto \ 
D_x^k.
\]
Thus $D_x^{9k} \sim D_x^k$, as desired.

\bigskip
\begin{figure}
    \centering
  $\underset{\textstyle\text{(a) $D_x^{-1}(y)$}}{\labellist\small\hair 2.5pt
        \pinlabel $b$ by 0 0 at 115 40
        \endlabellist   \includegraphics[scale=.6]{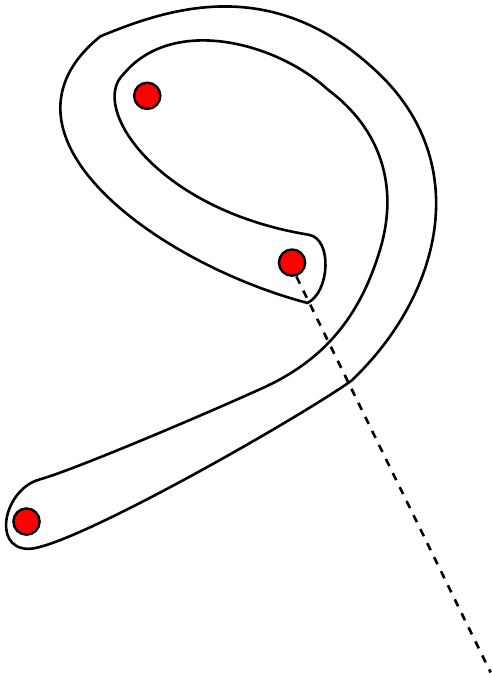}}$\qquad\qquad\qquad
    $\underset{\textstyle\text{(b) $R_3^{-1}(D_x^{-1}(y))$}}{\includegraphics[scale=.4]{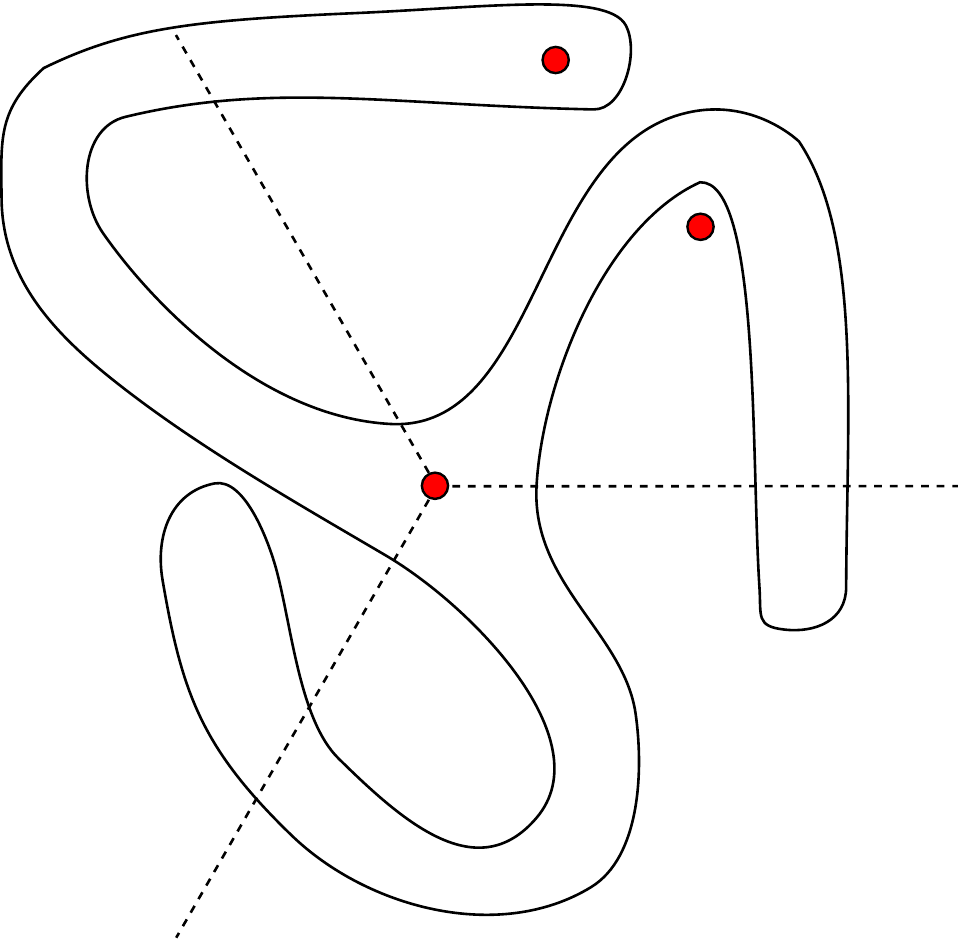}}$
    \caption{The curve $R_3^{-1}(D_x^{-1}(y))$ is homotopic to $z$.}
    \label{fig:D_x^-1(y)}
\end{figure}
 \noindent\emph{Case 2: $m=9k+1$.}  
 In this case we require one additional fact, namely that  $D^3_{D_x^{-1}(y)} \leadsto D_z$.  This follows from the fact that the preimage of the curve $D_x^{-1}(y)$ has a single component, which is isotopic to $z$; see Figure \ref{fig:D_x^-1(y)}.  We have
\[
D_x^{9k+1} \ 
\stackrel{D_x}{\leadsto} \ 
D_y^{3k}D_x \ 
\stackrel{D_x}{\leadsto} \ 
D_z^k D_x \ 
\stackrel{D_x}{\leadsto} \ 
\psi(D_{D_x^{-1}(z)}^k) D_x \ 
=
D_x.
\]

Thus, $D_x^{9k+1} \sim D_x$, as desired.

\bigskip

\noindent\emph{Case 3: $m=9k+2$.} In this case we use the fact that the preimage of the curve $D_x^{-2}(y)$ has a single component, which is isotopic to $z$.
\[ D_x^{9k+2} \ 
\stackrel{D_x^{2}}\leadsto \  D_y^{3k}D_x^2 \ 
\stackrel{D_x^2}{\leadsto} \ 
D_z^k D_x^2 \ 
\stackrel{D_x^2}{\leadsto} \ 
\psi(D_{D_x^{-2}(z)}^k) D_x^2 \ 
=
D_x^2.\]
Thus $D_x^{9k+2}\sim D_x^2$ as desired.
\bigskip

\noindent\emph{Case 4: $m=9k+3$.}  
We have 
\begin{align*}
D_x^{9k+3} \ 
{\leadsto} \ 
D_y^{3k+1} \ 
\stackrel{D_y}\leadsto \ 
D_z^k D_y \ 
\stackrel{D_y}\leadsto \ 
\psi(D_{D_y^{-1}(z)}^k) D_y \ 
=
D_y.
\end{align*}
Thus, $D_x^{9k+3} \sim D_y$, as desired.

\bigskip

\noindent\emph{Case 5: $m=9k+4$.}  In this case we have 
\begin{align*}
D_x^{9k+4} \ 
&\stackrel{D_x}{\leadsto} \
D_y^{3k+1}D_x=D_y^{3k}D_x^{-1}D_z^{-1}D_x = D_y^{3k}D_{D_x^{-1}(z)}^{-1} \ 
\leadsto \ 
D_z^k \ 
\leadsto
D_x^k,
\end{align*}
where we used the lantern relation in the first equality.  Then both $D_y^{3k}$ and $D_{D_x^{-1}(z)}^{-1}$ lift, where the latter lifts to the identity by Lemma~\ref{lem:triviality2.0}.  Thus, $D_x^{9k+4} \sim D_x^k$, as desired.

\bigskip

\noindent\emph{Case 6: $m=9k+5$.}  We use two facts that have not previously appeared.  First, the preimage of $D_x(y)$ is a single component that is homotopic to $D_y(z)$.  Second, the curve defining arc of $D_xD_y^2(z)$ has algebraic intersection 1 with $b$, therefore the preimage of $D_xD_y^2(z)$ is trivial by Lemma~\ref{lem:triviality2.0}.  We have
\begin{align*}
D_x^{9k+5} \ 
&\stackrel{D_x^{-1}}{\leadsto} \ 
D_y^{3k+2}D_x^{-1}   \stackrel{D_y^{-1}D_x^{-1}}{\leadsto} \ 
\psi(D_x D_y^{3k+3} D_x^{-1})D_y^{-1}D_x^{-1}=
(D_y D_z^{k+1} D_y^{-1}) D_y^{-1}D_x^{-1}\\
&\stackrel{D_y^{-1}D_x^{-1}}\leadsto
\psi(D_xD_y^2D_z^{k+1}D_y^{-2}D_x^{-1})D_y^{-1}D_x^{-1}=D_y^{-1}D_x^{-1}=
D_z
\leadsto
D_x
\end{align*}
where the final equality is an application of the lantern relation.
Thus, $D_x^{9k+5} \sim D_x$, as desired.

\bigskip

\noindent\emph{Case 7: $m=9k+6$.}  We have: 
\begin{align*}
D_x^{9k+6} \
{\leadsto} \ 
D_y^{3k+2} \ 
\stackrel{D_y^2}\leadsto \ 
D_z^k D_y^2 \ 
\stackrel{D_y^2}\leadsto \ 
\psi(D_{D_y^{-2}(z)}^k) D_y^2 \ 
=
D_y^2.
\end{align*}
Thus, $D_x^{9k+6} \sim D_y^2$, as desired.

\bigskip

\noindent\emph{Case 8: $m=9k+7$.}  In this case we use the fact that the preimage of $D_x^2(y)$ consists of single connected component, which is homotopic to $D_{x}^{-1}(z)$. We have:
\begin{align*}
D_x^{9k+7} \ 
\stackrel{D_x^{-2}}{\leadsto} \ 
D_y^{3k+3}D_x^{-2} \ 
\stackrel{D_x^{-2}}\leadsto \ 
D_{D_x^{-1}(z)}^{k+1}D_x^{-2} \ 
\stackrel{D_x^{-2}}\leadsto \ 
\psi(D_x^2 D_{D_x^{-1}(z)}^{k+1}D_x^{-2}) D_x^{-2} =
\psi(D_{D_x(z)}^{k+1})D_x^{-2}=D_x^{-2}.
\end{align*}
Thus, $D_x^{9k+7} \sim D_x^{-2}$, as desired.

\bigskip

\noindent\emph{Case 9: $m=9k+8$.}  In this case we use the fact that the preimage of the curve $D_x(y)$ consists of single connected component, which is homotopic to $D_y(z)$. We also apply the lantern relation to show that $D_xD_y(z)=z$.  We have
\begin{align*}
D_x^{9k+8} \ 
\stackrel{D_x^{-1}}{\leadsto} \
D_y^{3k+3} D_x^{-1} \ 
\stackrel{D_x^{-1}}\leadsto \ 
D_{D_y(z)}^{k+1} D_x^{-1} \ 
\stackrel{D_x^{-1}}\leadsto \ 
\psi(D_x D_{D_y(z)}^{k+1} D_x^{-1})D_x^{-1}=\psi(D_z^{k+1}) D_x^{-1}=
D_x^k.
\end{align*}
Thus, $D_x^{9k+8} \sim D_x^k$, as desired.

\subsection{Base cases} The second step in our proof of Theorem~\ref{thm:cubicrabbit} is to determine the polynomials to which the maps $D_xR_3$, $D_x^2R_3$, $D_yR_3$, $D_y^2R_3$, and $D_x^{-2}R_3$ are equivalent.  

The Hubbard trees for $R_3$ and $\overline{R}_3$ are tripods and the Hubbard trees for $A_3$ and $\overline{A}_3$ are paths of length 2. This can be see from their Julia sets, as shown in Figure~\ref{fig:generatorscubic}. As discussed in Section~\ref{sec:background}, the Hubbard tree, an invariant angle assignment on the tree, and the dynamical map on the edges of the tree induced by the polynomial, suffice to determine the polynomial.  The polynomial $R_3$ maps the edges of its Hubbard tree (a tripod) counterclockwise, while $\overline{R}_3$ maps the edges of its Hubbard tree clockwise. In both cases the angles at the leaves are $2\pi$ and the angles at the vertex of valence 3 each have measure $\frac{2\pi}{3}$. The polynomial $A_3$ supports an invariant angle assignment where the counterclockwise angle from edge $e_1$ and to edge $e_2$ is $2\pi/3$, while for $\overline{A}_3$ this angle is $4\pi/3$. In each case, $e_1$ is the edge from the critical point $0$ and the critical value. 

Therefore, in order to determine the equivalence class of $g R_3$ for $g \in \PMod(\mathbb{R}^2,P_{R_3})$, we first find a tree that is invariant under the lifting map. If it is a path of length 2 that supports an invariant angle assignment, then this is the topological Hubbard tree for $g R_3$ and the map is equivalent to either $A_3$ or $\overline{A}_3$; the invariant angle assignment distinguishes between the two possibilities. Moreover, the invariant angle assignment can be recovered from the full preimage: the critical value is a leaf of the topological Hubbard tree, and therefore the single angle adjacent to it has measure $2\pi$.  In the full preimage of the topological Hubbard tree, the critical point is trivalent, and each adjacent angle will have measure $\frac{2\pi}{3}$.  Because the critical point is bivalent in the topological Hubbard tree, one of the edges of the full preimage is not in the hull.  When this edge is removed, the angle on the side of the topological Hubbard tree from which the extra edge was removed will have angle $\frac{4\pi}{3}$. If the invariant tree for $g R_3$ is instead a tripod, then this is the topological Hubbard tree for the map and and the direction in which the dynamical map on the edges of the tree rotates the edges distinguishes between the two possibilities $R_3$ and $\overline{R}_3$.

\begin{figure}[ht]
\centering
$\underset{\textstyle\text{(a)}}{\labellist\small\hair 2.5pt
        \pinlabel $e_1$ by 0 0 at 70 15
        \pinlabel $e_2$ by 0 0 at 53 50
        \pinlabel $p_0$ by .5 .5 at 0 0
         \pinlabel $p_1$ by -1 1 at 116 58
        \pinlabel $p_2$ by 0 0 at 81 85
        \endlabellist
        \includegraphics[scale=.5]{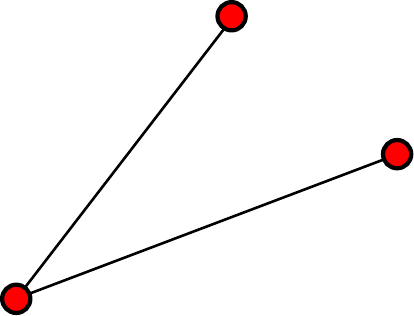}}$
\qquad\qquad
$\underset{\textstyle\text{(b)}}{\labellist\small\hair 2.5pt
        \pinlabel $e_1$ by 0 0 at 28 95
        \pinlabel $e_2$ by 0 0 at 53 50
        \pinlabel $p_0$ by .5 .5 at 0 0
         \pinlabel $p_1$ by -1 1 at 116 58
        \pinlabel $p_2$ by 0 0 at 81 85
        \endlabellist
        \includegraphics[scale=.5]{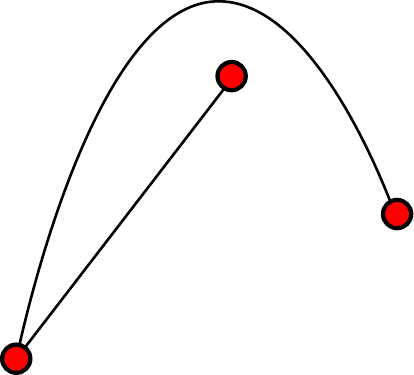}}$
\qquad\qquad
$\underset{\textstyle\text{(c)}}{\labellist\small\hair 2.5pt
        \pinlabel $e_1$ by 0 0 at 70 40
        \pinlabel $e_2$ by 0 0 at 95 90
        \pinlabel $p_0$ by .5 .5 at 0 0
         \pinlabel $p_1$ by -1 1 at 116 58
        \pinlabel $p_2$ by 1 0 at 60 85
        \endlabellist
        \includegraphics[scale=.5]{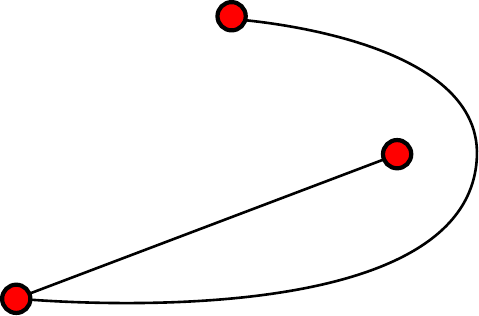}}$
\qquad\qquad
$\underset{\textstyle\text{(d)}}{\labellist\small\hair 2.5pt
        \pinlabel $p_0$ by .5 .5 at 0 0
         \pinlabel $p_1$ by -1 1 at 85 58
        \pinlabel $p_2$ by 1 0 at 60 85
        \endlabellist\includegraphics[scale=.5]{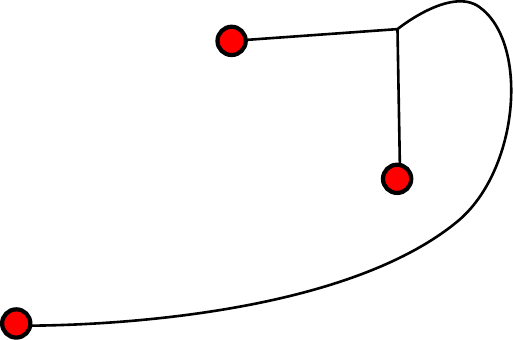}}$

\caption{(a)~The topological Hubbard tree for $D_xR_3$ and $D_x^2R_3$. (b)~The topological Hubbard tree for $D_yR_3$ and $D_y^2R_3$. (c)~The topological Hubbard tree for $D_x^{-2}R_3$. (d)~The topological Hubbard tree for $D_x^{-1}R_3$.}
\label{fig:twistedcubictrees}
\end{figure}
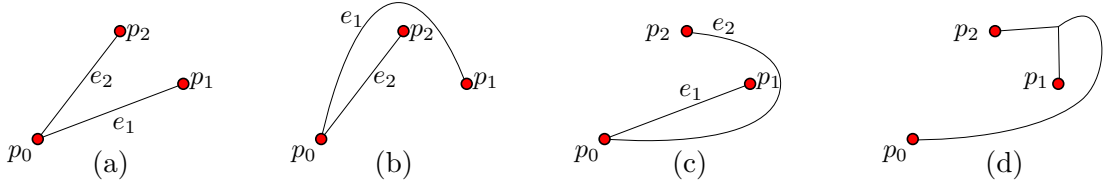

\begin{proof}[Proof of Theorem~\ref{thm:cubicrabbit}]
If $m\geq 0$ and the 9-adic expansion of $m$ contains only 0's, 4's, and 8's, then the reduction formulas reduce $D_x^m R_3$ to $R_3$. 

If $m<0$ and the 9-adic expansion of $m$ contains only 0's, 4's, and 8's, the reduction formulas reduce $m$ to $\dots 888=-1$.  So in this case we have $D_x^m R_3 \simeq D_x^{-1}R_3$. The topological Hubbard tree for $D_x^{-1}R_3$ is shown in Figure~\ref{fig:twistedcubictrees}. Since $D_x^{-1}R_3$ maps the edges of this tripod clockwise (relative to the trivalent point), we have $D_x^{-1}R_3 \simeq \overline{R}_3$.

If there exists a digit in the 9-adic expansion of $m$ that is not 0, 4, or 8, let $\dig$ be the first such digit.  Then the reduction formulas determine that $D_x^mR_3$ is equivalent to one of: $D_xR_3$, $D_x^2R_3$, $D_yR_3$, $D_y^2R_3$, or $D_x^{-2}R_3$, depending on $\dig$.  The topological Hubbard tree for each of these base cases is depicted in Figure~\ref{fig:twistedcubictrees}.  Each of these topological Hubbard trees is a path of length 2, so the corresponding maps are equivalent to either $A_3$ or $\overline{A}_3$. Further, the trees for $D_xR_3$ and $D_y^2R_3$ support an invariant angle assignment with angle $\angle(e_1,e_2)=2\pi/3$, and so $D_xR_3$ is equivalent to $A_3$. Similarly, the trees for $D_x^2 R_3,\, D_yR_3$ and $D_x^{-2}R_3$ support an angle structure with angle $\angle(e_1,e_2)=4\pi/3$, and so $D_x^2 R_3$ is  equivalent to $\overline{A}_3$. (Angles are measured counterclockwise from $e_1$ to $e_2$).
\end{proof}

\section{The general solution for the twisted cubic rabbit}
\label{sec:algo}

In this section we describe two closely related algorithms for determining the equivalence class of $g R_3$, where $g$ is an arbitrary pure mapping class in $\PMod(\mathbb{R}^2,P_{R_3})$. We first give a ``whole word" algorithm that is directly analogous to the one given by  Bartholdi--Nekrashevych for the quadratic rabbit \cite[Section 4.1]{BaNe}. The proof we give for this algorithm follows their approach via the theory of self-similar groups. Second, we give a ``prefix" algorithm and justify it through an elementary argument using word length. Finally, we give some computational results of applying these algorithms.

\subsection{Whole-word algorithm} The pure mapping class group $\PMod(\mathbb{R}^2,P_{R_3})$ is a free group generated by $D_x$ and $D_z$ as a free basis. There is an index 3 subgroup  $\LMod(\R^2,P_{R_3})< \PMod(\mathbb{R}^2,P_{R_3})$ consisting of elements that are liftable through $R_3$ to pure mapping classes. This subgroup is generated by the following elements:

\[ \mathcal{L} = 
\{D_x^3, \ D_z, \ D_x^{-1}D_zD_x, \ D_xD_zD_x^{-1}\}
\]

In the previous section we showed the following facts:

\[
D_x^3 \leadsto D_y=D_x^{-1}D_z^{-1} \ \ \ \ D_z \leadsto D_x \ \ \ \ D_x^{-1}D_zD_x \leadsto id \ \ \ \ D_xD_zD_x^{-1} \leadsto id.
\]
\vskip.15in
\noindent These facts yield a homomorphism $\psi: \LMod(\mathbb{R}^2,P_{R_3}) \to \PMod(\mathbb{R}^2,P_{R_3})$.

We can extend the homomorphism $\psi$ to a well-defined set map $\overline{\psi}$ from $\PMod(\mathbb{R}^2,P_{R_3})$ to itself as follows:

\[
\overline{\psi}: g \mapsto 
\begin{cases}
\psi(g),& g \in \LMod(\mathbb{R}^2,P_{R_3})\\

\psi(D_x^{-1}g)D_x,& D_x^{-1}g \in \LMod(\mathbb{R}^2,P_{R_3})\\

\psi(D_xg)D_x^{-1},& D_xg\in \LMod(\mathbb{R}^2,P_{R_3})\\
\end{cases}
\]\\

\noindent By \cite[Lemma 5.1]{BLMW}, we have: that $gR_3 \simeq \overline{\psi}(g)R_3$ for all $g\in \PMod(\mathbb{R}^2,P_{R_3})$. 

\p{Wreath recursions} Following Bartholdi--Nekrashevych \cite[Proposition~4.2]{BaNe}, we encode the map $\bar{\psi}$ using a wreath recursion; see their Section~2.2 for additional background. Note that our convention of function composition notation runs opposite to their group theoretic notation. A wreath recursion is a homomorphism $\Phi: G \to G \wr \Sigma_n$ where $G$ is a group and $\Sigma_n$ is a symmetric group acting on $X=\{1,\dots,n\}$. We write elements of the wreath product in the form $\sigma \langle\langle g_n,\dots,g_1 \rangle\rangle $, with $\sigma \in \Sigma_n$ and $g_i \in G$. If $\sigma$ is the identity in $\Sigma_n$ or if all of the $g_i$ are the identity in $G$, these may be suppressed in the notation. Multiplication in the wreath recursion is carried out through two rules. First, two adjacent elements in angle brackets are multiplied in the group $G$ entrywise: $\langle\langle h_n,\dots,h_1 \rangle\rangle \langle\langle g_n,\dots,g_1 \rangle\rangle = \langle\langle h_ng_n,\dots,h_1g_1 \rangle\rangle$. Second, an element $\sigma$ acts by permutation on the indices of the $g_i$ when pushed past an angle bracket term: $\langle\langle g_n,\dots,g_1 \rangle\rangle \sigma=\sigma \langle\langle g_{\sigma(n)},\dots,g_{\sigma(1)} \rangle\rangle$. For a wreath recursion $\Phi$ and $g \in G$, the restriction map $g|_i$ is the $i$th coordinate of $\Phi(g)$. Restriction maps can be composed; inductively we put $g|_{xv}= (g|_v)|_x$ for all $x \in X$ and $v \in X^*$, the set of finite words in $X$.

A benefit of this coding is that wreath recursions have a simple criterion that determines when there exists a finite set to which the restriction map contracts under iteration (i.e. a nucleus): a self-similar action of a group $G$ with finite symmetric generating set $S$ (with $1 \in S$) is contracting if and only if there exists a finite set $\mathcal{N} \subset G$ and a number $k \in \mathbb{N}$ such that $((S \cup \mathcal{N})^2)|_{X^k} \subseteq \mathcal{N}$ \cite[Lemma~2.11.2]{SSG}. The GAP package AutomGrp can be used to compute the nucleus of a contracting wreath recursion \cite{AutomGrp1.3.2}. 

\begin{theorem}
Iterating $\overline{\psi}$ on an element $g \in \PMod(\mathbb{R}^2,P_{R_3})$ yields exactly one member of the following set: $\{id, D_x, D_x^{-1}, D_zD_x^2\}$. The equivalence class of $g R_3$ is then determined as follows:
\[
g R_3 \simeq
\begin{cases}
R_3,&\text{ if $\overline{\psi}^n(g)=id$}\\
\overline{R}_3,&\text{ if $\overline{\psi}^n(g)=D_x^{-1}$}\\
A_3,&\text{ if $\overline{\psi}^n(g)=D_x$}\\
\overline{A}_3,&\text{  if $\overline{\psi}^n(g)=D_zD_x^2$}\\
\end{cases}
\]
\end{theorem}

\begin{proof}
Consider the wreath recursion $\Phi: \PMod(\mathbb{R}^2,P_{R_3}) \to \PMod(\mathbb{R}^2,P_{R_3}) \wr \Sigma_3$ given by
\[
\Phi(D_x)=\rho \langle\langle D_x^{-1}D_z^{-1},id,id \rangle\rangle  \ \ \ \ \text{   and   } \ \ \ \ \Phi(D_z)=\langle\langle id,id,D_x \rangle\rangle
\]

\noindent where $\rho$ is the permutation $(132)$: $\rho(1)=3$, $\rho(3)=2$, $\rho(2)=1$.

This wreath recursion encodes the map $\overline{\psi}$; computing $\overline{\psi}(g)$ is the same as computing $g|_1$, and then adjusting the result according to the value of the accompanying permutation factor. For $g \in \mathcal{L}$, we have that $\bar{\psi}(g)=g|_1$ and that $\Phi(g)$ has trivial permutation factor; since $\Phi$ is a homomorphism, the same is true for all $g \in \LMod(\mathbb{R}^2,P_{R_3})$. In the case where $D_x^{-1}g \in \LMod(\mathbb{R}^2,P_{R_3})$, we have that $\Phi(g)= \rho \langle \langle g_3,g_2,g_1 \rangle\rangle$ and
\[
\bar{\psi}(g)=\psi(D_x^{-1}g)D_x=(D_x^{-1}g)|_1D_x=
(\rho^2 \langle \langle id, D_z D_x,id \rangle \rangle \rho \langle \langle g_3,g_2,g_1 \rangle\rangle)|_1 D_x=
g|_1 D_x
\]
Similarly, in the case where $D_x g \in \LMod(\mathbb{R}^2,P_{R_3})$, we have $\bar{\psi}(g)=g|_1 D_x^{-1}$. Since $D_x|_1=D_x^{-1}|_1=id$, we have by induction for all $n \in \mathbb{N}$ that 
\[\bar{\psi}^n(g)=g|_v \text{ \ \  or  \ \  } \bar{\psi}^n(g)=g|_v D_x \text{ \ \  or  \ \  } \bar{\psi}^n(g)=g|_v D_x^{-1} \text{  \ \  for some \ \ } v \in X^n.
\]

A small computation yields the following nucleus set for the restriction map for $\Phi$:
\[
\mathcal{N} = \{id, \ D_x, \ D_x^{-1}, \  D_zD_x, \ D_x^{-1}D_z^{-1}\}
\]

\noindent Therefore for all $g \in \PMod(\mathbb{R}^2,P_{R_3})$ we have that  $\bar{\psi}^n(g) \in \mathcal{N} \cup \mathcal{N}D_x^{-1} \cup \mathcal{N}D_x$ for all sufficiently large $n$. It then suffices to analyze the dynamics of $\overline{\psi}$ acting on this finite set; the only cycles are the fixed points $id$, $D_x$, and $D_x^{-1}$ and the 2-cycle on $D_zD_x^2$ and $D_x^{-1}D_z^{-1}D_x^{-1}$. Taking $D_zD_x^2$ as the representative for the 2-cycle yields the four elements $\{id, D_x, D_x^{-1}, D_zD_x^2\}$ in the theorem statement. The computations of the base cases in Section \ref{sec:cubic} then yield the result, since $D_xR_3 \simeq A_3$ and $D_x^{-1}R_3 \simeq \overline{R}_3$ are among these, while $D_zD_x^2 \sim D_x^{-2}$ and $D_x^{-2}R_3 \simeq \overline{A}_3$.
\end{proof}

\subsection{Prefix algorithm} Any $g \in \PMod(\mathbb{R}^2,P_{R_3})$ of reduced word length at least 4 in the free generating set $\{D_x, \ D_z \}$ can be subdivided into two pieces $g=hp$ where $h$ is possibly the empty word and $p$ is one of the eight prefixes (or inverted variants thereof) as shown in Table~\ref{tab:prefixes}.

\begin{center}
\begin{table}[h]
\begin{tabular}{||c l l l||} 
 \hline
Case & $g=hp$ & $g'$ & $\Delta=|g'|-|g|$\\ [0.5ex] 
 \hline\hline
1) & $hD_z$ & $D_xh$ & $\Delta \leq 0$ \\ 

2) & $hD_x^3$ & $D_x^{-1}D_z^{-1}h$ & $\Delta \leq -1$ \\ 

3) & $hD_x^{-1}D_zD_x$ & $h$ & $\Delta = -3$ \\ 

4) & $hD_xD_zD_x^{-1}$ & $h$ & $\Delta = -3$ \\ 

5) & $hD_zD_x=hD_xD_x^{-1}D_zD_x$ & $hD_x$ & $\Delta = -1$ \\ 

6) & $hD_xD_zD_xD_x=hD_xD_zD_x^{-1}D_x^3$ & $D_x^{-1}D_z^{-1}h$ & $\Delta \leq -2$ \\

7) & $hD_x^{-1}D_zD_xD_x=hD_xD_x^{-2}D_zD_x^2$ & $hD_x$ & $\Delta \leq -3$ \\ 

8) & $hD_zD_zD_xD_x=hD_x^{-1}D_xD_zD_zD_x^{-1}D_x^3$ & $ D_x^{-1}D_z^{-1}hD_x^{-1}$ & $\Delta \leq -1$\\

 \hline
\end{tabular}
\caption{ The eight cases of prefixes for words in $\PMod(\mathbb{R}^2,P_{R_3})$ of length at least 4 and the effects of lifting these through $R_3$ on reduced word length. }\label{tab:prefixes}
\end{table}
\end{center}

In each case the prefix $p$ may be lifted through $R_3$ (perhaps with borrowing) and then the lift may appended to $h$. This yields a new word $g'$, and by the same logic as \cite[Lemma 5.1]{BLMW} we have that $g'R_3 \simeq gR_3$. In the first four cases $p$ itself is liftable; each is a generator in $\mathcal{L}$. In the sixth case $p$ can be rewritten as a product of two of these generators. In the remaining three cases borrowing is required in order to lift $p$. The change in reduced word length in each case is recorded in Table~\ref{tab:prefixes}; the fact that in some cases we have an inequality for this change comes from the fact that the new word may admit free reduction.

Let $P(g)=g'$ be the prefix lifting map just described, which is well defined on all $g \in \PMod(\mathbb{R}^2,P_{R_3})$ of reduced word length at least 4, as well as for some shorter words. For the remaining reduced words in $g \in \PMod(\mathbb{R}^2,P_{R_3})$ of length at most 3, we define $P(g)=g$. We do this for the sake of making the algorithm as simple as possible, although the trade-off is that there are more terminal words than is necessary. (The issue is that further lifting shows that some of these short words are in the same equivalence class, but this further lifting does not necessarily decrease reduced word length.) We now show that iterating the map $P$ yields an algorithm for determining the equivalence class of $gR_3$.

\begin{theorem}
For any $g \in \PMod(\mathbb{R}^2,P_{R_3})$, there exists a $k \geq 0$ such that $|P^k(g)| < 4$. Moreover, for any $g$, there exists an $k$ such that $P^k(g)$ is one of the following nine values, and the equivalence class of $gR_3$ is determined by the corresponding value in Table~\ref{tab:terminal}.
\end{theorem}

\begin{center}
\begin{table}[!ht]
\begin{tabular}{||r |c |c |c |c |c |c |c |c |c ||} 
 \hline
 \rule{0pt}{18pt}
 $P^k(g)=$ & $id$ & $D_x$ & $D_x^{-1}$ & $D_x^2$ &  $D_x^{-2}$ &  $D_zD_x^2$ & $D_zD_x^{-2}$ &  $D_z^{-1}D_x^2$ & $D_z^{-1}D_x^{-2}$\\
 
 \hline
\rule{0pt}{18pt}$gR_3 \simeq $ & $R_3$ & $A_3$ & $\overline{R}_3$ & $\overline{A}_3$ & $\overline{A}_3$ & $\overline{A}_3$ & $\overline{A}_3$ & $\overline{A}_3$ & $\overline{A}_3$\\
 \hline
\end{tabular}
\caption{The nine terminal words under iterates of $P$ on $\PMod(\mathbb{R}^2,P_{R_3})$ and the corresponding equivalence classes of $gR_3$.}
\label{tab:terminal}
\end{table}
\end{center}

\begin{proof}
If $|g|<4$, the first statement holds trivially. Then let $g \in \PMod(\mathbb{R}^2,P_{R_3})$ be a reduced word of length at least 4. We show that for some $k>0$ we have $|P^k(g)|<|g|$; the first statement will then hold by induction. If the prefix of $g$ falls into one of the cases 2 through 8 in Table~\ref{tab:prefixes}, we have that $\Delta \leq -1$ and so we have the desired decrease in reduced word length with $k=1$. Otherwise, we are in Case 1 and $g=hD_z$ (or, similarly, $g=hD_z^{-1}$). Then $g'=D_xh$ and reduced word length may not have decreased, but it has not increased, and now there is a new prefix. Further applications of $P$ do not increase word length and eventually produce a word that begins with $D_x$ or $D_x^{-1}$. Note that every word $g'$ in Table~\ref{tab:prefixes} begins with either $h$ or $D_x$ or $D_x^{-1}$, and that even in the situation where $g$ is a power of $D_z$, applications of $P$ eventually produce a power of $D_x$. Thus Case 1 reduces to the other cases.

It remains to consider the dynamics of $P$ on words of length at most 3. The only cycles are the $9$ fixed points recorded in Table \ref{tab:terminal} along with the corresponding equivalence classes of $gR_3$. Again, these were either computed as base cases in Section~\ref{sec:cubic} or (in the case of the five words yielding $\overline{A}$) are shown to be equivalent to a base case with a small amount of direct calculation.\end{proof}

\p{Computations} Recall that the group $\PMod(\mathbb{R}^2,P_{R_3})$ is a free group of rank 2. With the generating set $\{ D_x, D_z \}$ it contains $2 \cdot 3^{\ell}-1$ elements of reduced word length at most $\ell$. In Table~\ref{tab:counts} we give the counts of these elements according to their equivalence classes $gR_3$ for small values of $\ell$. We have computed these results using both the whole-word and prefix algorithms. Note that the tallies given in the chart are dependent on the chosen generating set.

\begin{center}
\begin{table}[h!]
\begin{tabular}{||r |r| r |r |r |r |r |r |r |r |r||} 
 \hline
\rule{0pt}{16pt}$\ell$        & $0$ & $1$ & $2$ & $3$ & $4$ &  $5$ &  $6$ & $7$ &  $8$ & $9$\\
 \hline
 \hline
\rule{0pt}{16pt} $R_3$        & $1$ & $1$ & $3$ & $11$ & $27$ &  $94$ &  $287$ & $857$ &  $2527$ & $7341$\\
 \hline
\rule{0pt}{16pt} $\overline{R}_3$  & $0$ & $2$ & $4$ & $8$ & $29$ &  $82$ &  $258$ & $785$ &  $2294$ & $6802$\\
 \hline
\rule{0pt}{16pt}$A_3$       & $0$ & $2$ & $4$ & $12$ & $48$ &  $139$ &  $445$ & $1367$ &  $4078$ & $12495$\\
 \hline
\rule{0pt}{16pt}$\overline{A}_3$ & $0$ & $0$ & $6$ & $22$ & $57$ &  $170$ &  $467$ & $1364$ &  $4222$ & $12727$\\
 \hline
\end{tabular}
\caption{The relative frequency of $R_3$, $\overline{R}_3$, $A_3$ and $\overline{A}_3$ for $gR_3$ with $|g|\leq \ell$ in the generating set $\{D_x,D_z\}$}
\label{tab:counts}
\end{table}
\end{center}

It would be interesting to know whether there is a limiting ratio among the equivalence classes as $\ell$ goes to infinity, and if so, what this ratio is. To our knowledge an exact result along these lines is not known even in the case of the quadratic rabbit. Answers to these questions would represent a significant step forward in the study of twisted polynomial problems. 

\section{Twisting the many-eared cubic rabbit}
\label{sec:cubicn}

We now turn to many-eared cubic rabbits $R_n$, the analogues of the cubic rabbit but where the critical portrait is a cycle of $n$ marked points. We begin by giving a combinatorial description of the maps $R_n$, just as we did for $R_3$. We then describe the families of topological polynomials that arise in Theorem~\ref{thm:many}, the solution of the twisted many-eared cubic rabbit problem. We then solve this problem in two steps, first producing reduction formulas and then determining the base cases.

The Hubbard tree for $R_n$ is depicted in Figure~\ref{fig:R6}. Labeling the edges so that $e_i$ has $p_i$ as one of its endpoints, the induced map of $R_n$ on the edges of its Hubbard tree is:
$$(R_n)_*(e_i) = 
\begin{cases}
e_{i+1} & 0\leq i \leq n-1 \\
\end{cases}$$
with indices taken mod $n$. The complex kneading sequence for the family $R_n$ is $\overline{11\cdots1\ast}$.  
The polynomials in the family $\overline{R}_n$ also have $n$-pods as their Hubbard trees, but $\overline{R}_n$ rotates the edges of these $n$-pods clockwise instead of counterclockwise.  The complex kneading sequence for $\overline{R}_n$ is also $\overline{11\cdots 1\ast}$.

We now give a combinatorial topology description of a map that is homotopic to $R_n$ relative to $P_{R_n}$, and so also equivalent to $R_n$. For the post-critical set $P_{R_n}$, let $0=p_0$ be the origin and let the points $p_i$ be $R_n^i(0)$ for $1 \leq i \leq n-1$. Let $\Delta_n$ denote the solid $n$-gon with vertex set $P_{R_n}$. Let $\Cub$ be any triple branched cover $(\R^2,P_{R_n})\to(\R^2,P_{R_n})$ that is only branched over 0 and fixes pointwise $\Delta_n$.  Let $\Rot_n$ be a homeomorphism of $(\R^2,P_{R_n})$ that rotates the points $P_{R_n}$ counterclockwise, with each $p_i$ mapping to $p_{i+1}$, and that preserves $\Delta_n$ setwise. This homeomorphism is unique up to homotopy relative to $P_{R_n}$. Then $\Rot_n \circ \Cub$ is homotopic to $R_n$, for each value of $n$, as follows.  The $n$-pod contained in $\Delta_n$ is invariant under lifting by $\Rot_n\circ \Cub$ and $\Rot_n\circ \Cub$ permutes the edges counterclockwise.  Moreover, the $n$-pod satisfies Poirier's conditions to be the topological Hubbard tree for $\Rot_n\circ \Cub$. Since the Hubbard tree for $R_n$ is a homotopic $n$-pod relative to $P_{R_n}$ and $R_n$ permutes the edges counterclockwise, $R_n$ and $\Rot_n\circ\Cub$ are homotopic. This is again an application of the Alexander method.

For any fixed $n$, let the curve $c_i$ be the boundary of a regular neighborhood of the straight arc from $p_{i-1}$ to $p_i$ with indices taken mod $n$. Set $x=c_2$, $y=c_1$, and $z=c_0$. These curves are also depicted in Figure~\ref{fig:R6}.

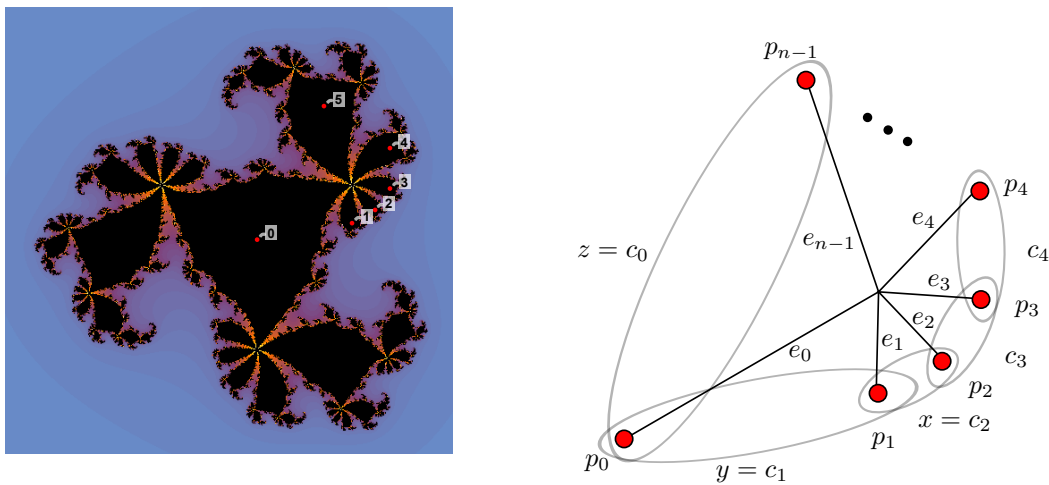
\begin{figure}
    \centering
    \includegraphics[width=.4\textwidth]{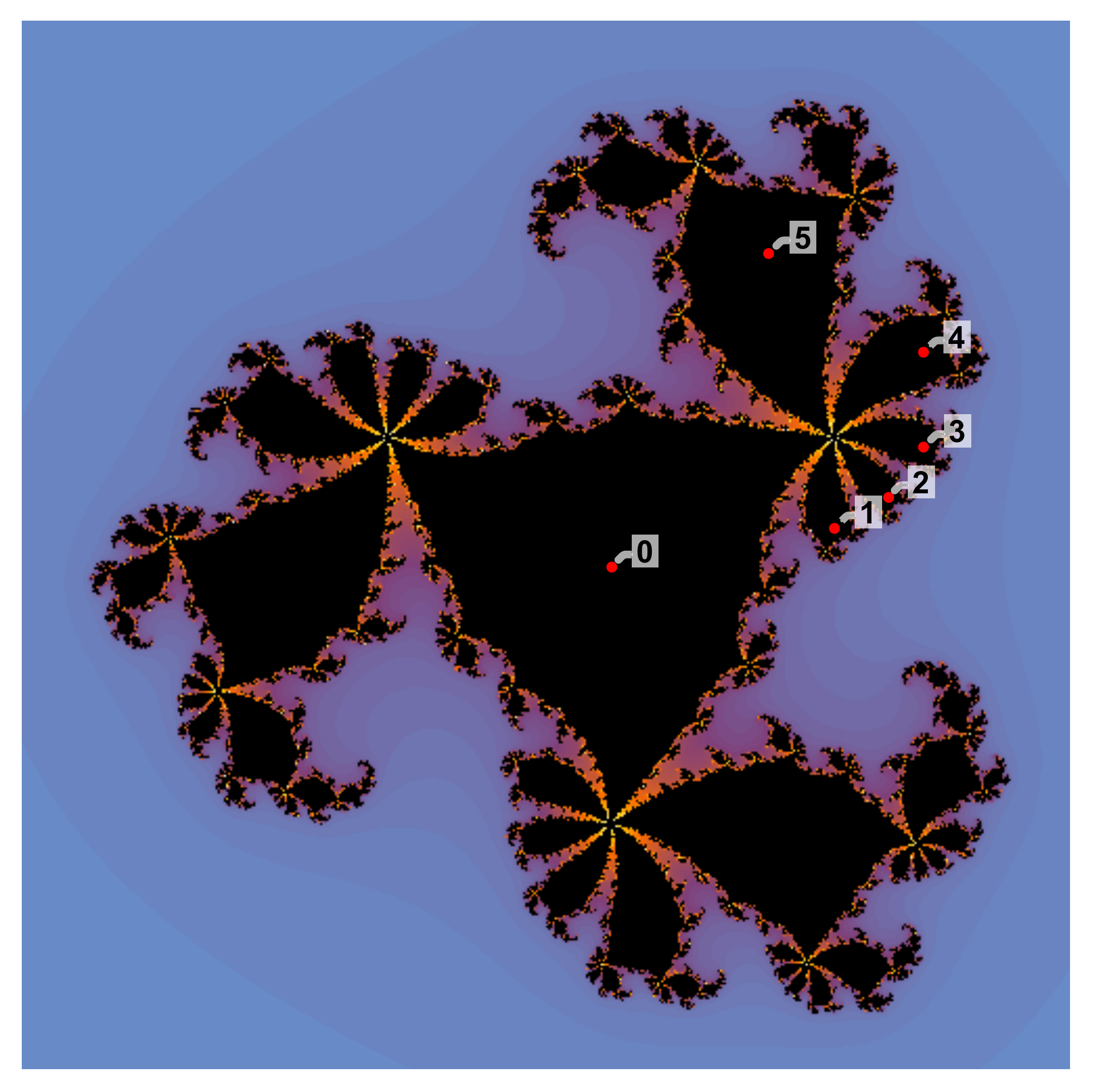}\qquad \qquad \ \ 
    \labellist\small\hair 2.5pt
        \pinlabel {$p_0$} by 0 0 at 0 0
        \pinlabel {$p_1$} by 0 0 at 130 10
         \pinlabel {$p_2$} by 0 0 at 174 33
           \pinlabel {$p_3$} by 0 0 at 195 70
            \pinlabel {$p_4$} by 0 0 at 190 125
             \pinlabel {$p_{n-1}$} by 0 0 at 88 187
        \pinlabel {$y=c_1$} by 0 0 at 70 -5
        \pinlabel {$x=c_2$} by 0 0 at 162 18
        \pinlabel {$c_3$} by 0 0 at 190 47
        \pinlabel {$c_4$} by 0 0 at 200 95
         \pinlabel {$z=c_0$} by 0 0 at 7 95
        \pinlabel {$e_0$} by 0 0 at 92 50
        \pinlabel {$e_1$} by 0 0 at 134 55
        \pinlabel {$e_2$} by 0 0 at 148 65
        \pinlabel {$e_3$} by 0 0 at 155 82
        \pinlabel {$e_4$} by 0 0 at 148 110
         \pinlabel {$e_{n-1}$} by 0 0 at 105 100
        \endlabellist
    \includegraphics[width=.35\textwidth]{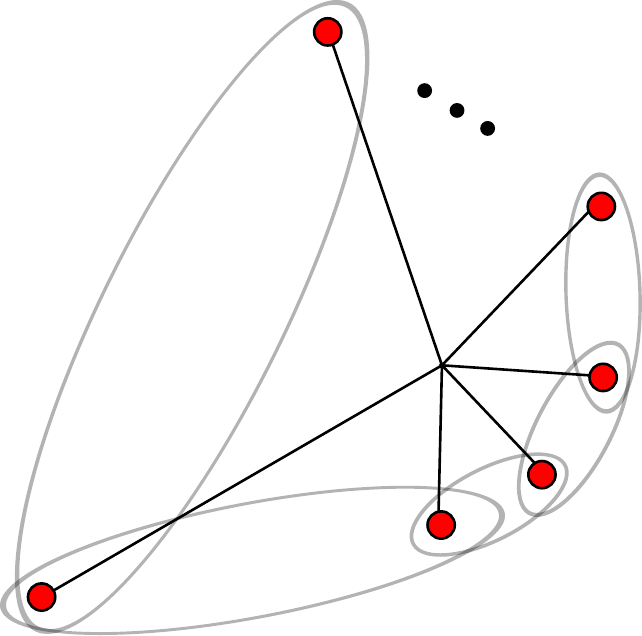}
    \caption{The Julia set for $R_6$. The Hubbard tree and the curves $c_i$ for $R_n$.}
    \label{fig:R6}
\end{figure}

\subsection{The required polynomial families} We now describe the polynomials that appear in the statement of Theorem~\ref{thm:many}. These polynomials fall in nine families, three of which are in complex conjugate pairs. That is, one family in the pair is a sequence of polynomials $z^3+c_n$ while the other is a sequence of polynomials $z^3+\overline{c_n}$. While the corresponding maps in the families are not equivalent, they are conjugate via the orientation-reversing homeomorphism of complex conjugation. Therefore it suffices to give descriptions in detail of six families.  For each of the families, we provide multiple pieces of identifying information: the kneading sequence, the Hubbard tree and the dynamical map on the edges, the angle at the critical point, and the direction of rotation of the post-critical set. This data is redundant, but we hope that readers with different mathematical tastes will find their preferred descriptions. Note that the Hubbard trees are depicted schematically. The complex kneading sequence, defined by Hubbard--Schleicher \cite{spider}, is in the cubic case a string of digits in the set $\{0,1,2,\ast\}$ that describes the {\it itinerary} of the post-critical points relative to the unit circle. (See Kauko \cite{kauko} for a definition of higher-degree kneading sequences and details of the convention we use). In what follows, we we name the critical point $p_0$ and $p_i=f^i(p_0)$ for $1\leq i\leq n-1$.

Five of the families of polynomials that arise in our solution are cubic analogues of families that appeared in the quadratic twisted many-eared rabbit problem.  The first of the families is $R_n$ and we will call the other four $A_n$, $B_n$, $K_{n,1}$ and $K_{n,2}$, as in Belk--Lanier--Margalit--Winarski \cite[Section 5]{BLMW}. The maps $K_{n,1}$ and $K_{n,2}$ are both generalizations of the quadratic Kokopelli family $K_n$.  As in the quadratic case, the Hubbard trees for the families $A_n$ and $B_n$ are paths of length $n-1$.  Unlike the quadratic case, the families $A_n$ and $B_n$ do not have real coefficients; note that their Hubbard trees do not lie in the real line and their post-critical sets do not lie on a line, although we depict them in this way. 

\begin{figure}
    \centering
    \subcaptionbox{The Julia set for $A_6$ and the Hubbard tree for $A_n$\label{fig:A6}}{$\vcenter{\hbox{    \includegraphics[width=.3\textwidth]{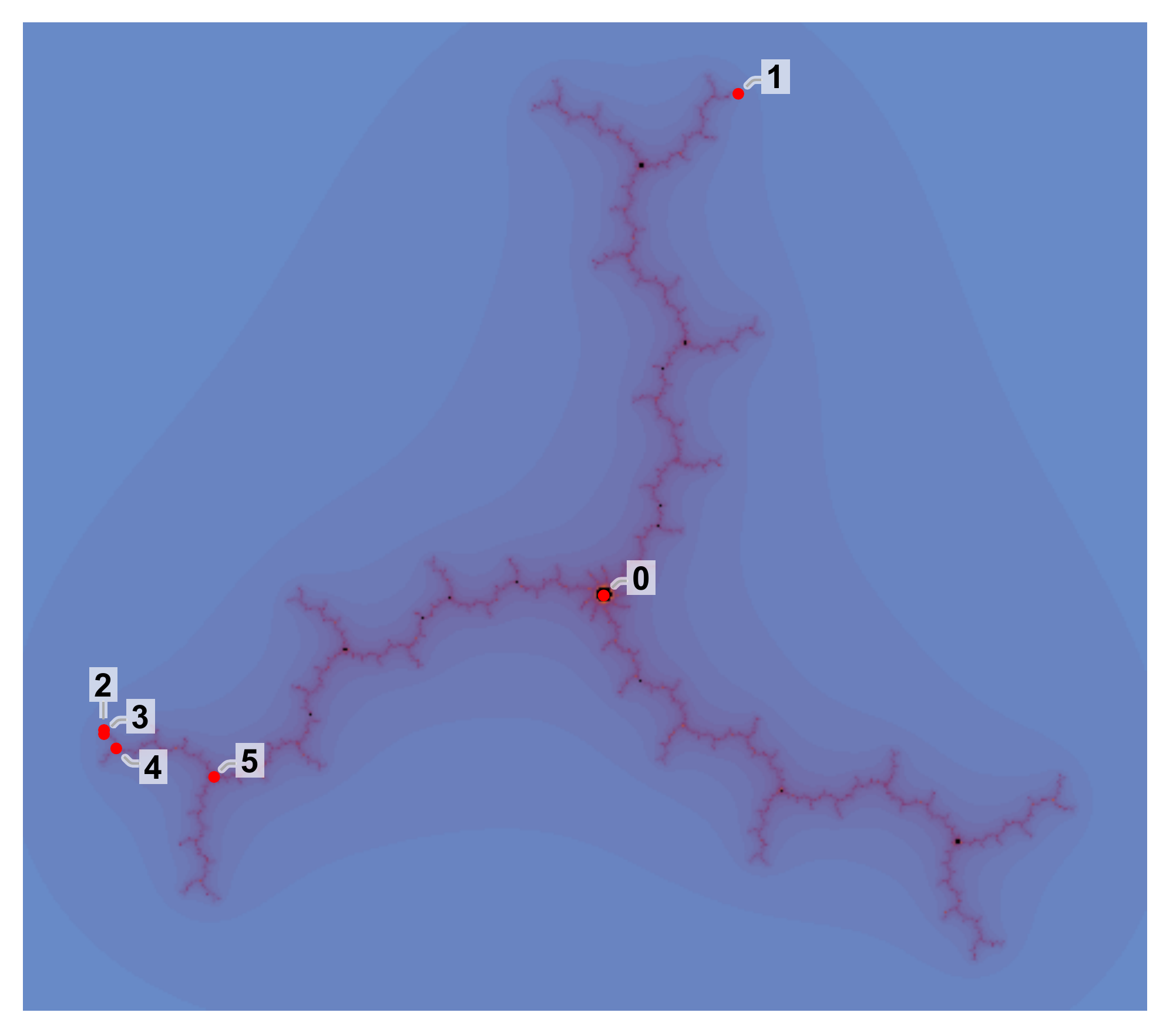}}}$\ \ \ \ \

  $\vcenter{\hbox{{\labellist\small\hair 2.5pt
        \pinlabel {$p_1$} by 0 1 at 0 0
        \pinlabel {$p_0$} by 0 1 at 60 0
         \pinlabel {$p_{n-1}$} by 0 1 at 120 0
          \pinlabel {$p_{n-2}$} by 0 1 at 177 0
           \pinlabel {$p_2$} by 0 1 at 260 0
              \pinlabel {$e_1$} by 0 -1.5 at 30 0
          \pinlabel {$e_2$} by 0 -1.5 at 87 0
           \pinlabel {$e_3$} by 0 -1.5 at 140 0
        \endlabellist\includegraphics[width=.3\textwidth]{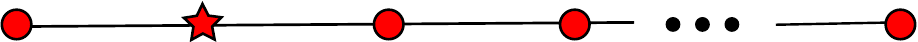}}}}$}
    
   \subcaptionbox{The Julia set for $B_6$ and the Hubbard tree for $B_n$\label{fig:B6}}{    $\vcenter{\hbox{\includegraphics[width=.3\textwidth]{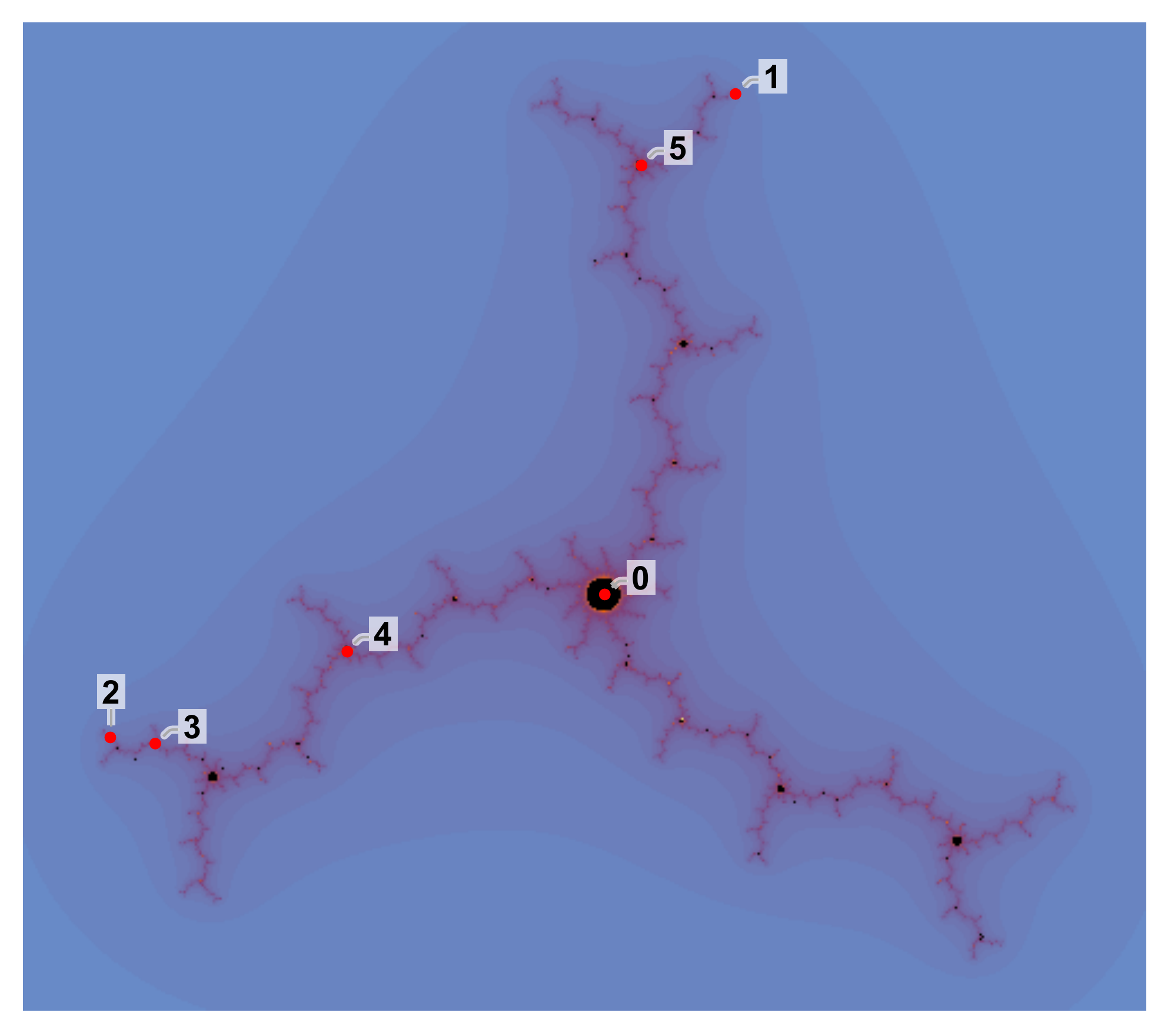}}}$\ \ \ \ 
    $\vcenter{\hbox{\labellist\small\hair 2.5pt       
    \pinlabel {$p_1$} by 0 1 at 0 0
        \pinlabel {$p_{n-1}$} by 0 1 at 60 0
         \pinlabel {$p_0$} by 0 1 at 120 0
          \pinlabel {$p_{n-2}$} by 0 1 at 177 0
           \pinlabel {$p_2$} by 0 1 at 260 0
              \pinlabel {$e_1$} by 0 -1.5 at 30 0
          \pinlabel {$e_2$} by 0 -1.5 at 87 0
           \pinlabel {$e_3$} by 0 -1.5 at 140 0
        \endlabellist\includegraphics[width=.3\textwidth]{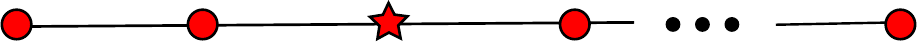}}}$}

    \subcaptionbox{The Julia set for $K_{6,1}$.  The Hubbard tree for $K_{n,1}$ and $K_{n,2}$. The Julia set for $K_{6,2}$.\label{fig:K6}}{    $\vcenter{\hbox{\includegraphics[width=.3\textwidth]{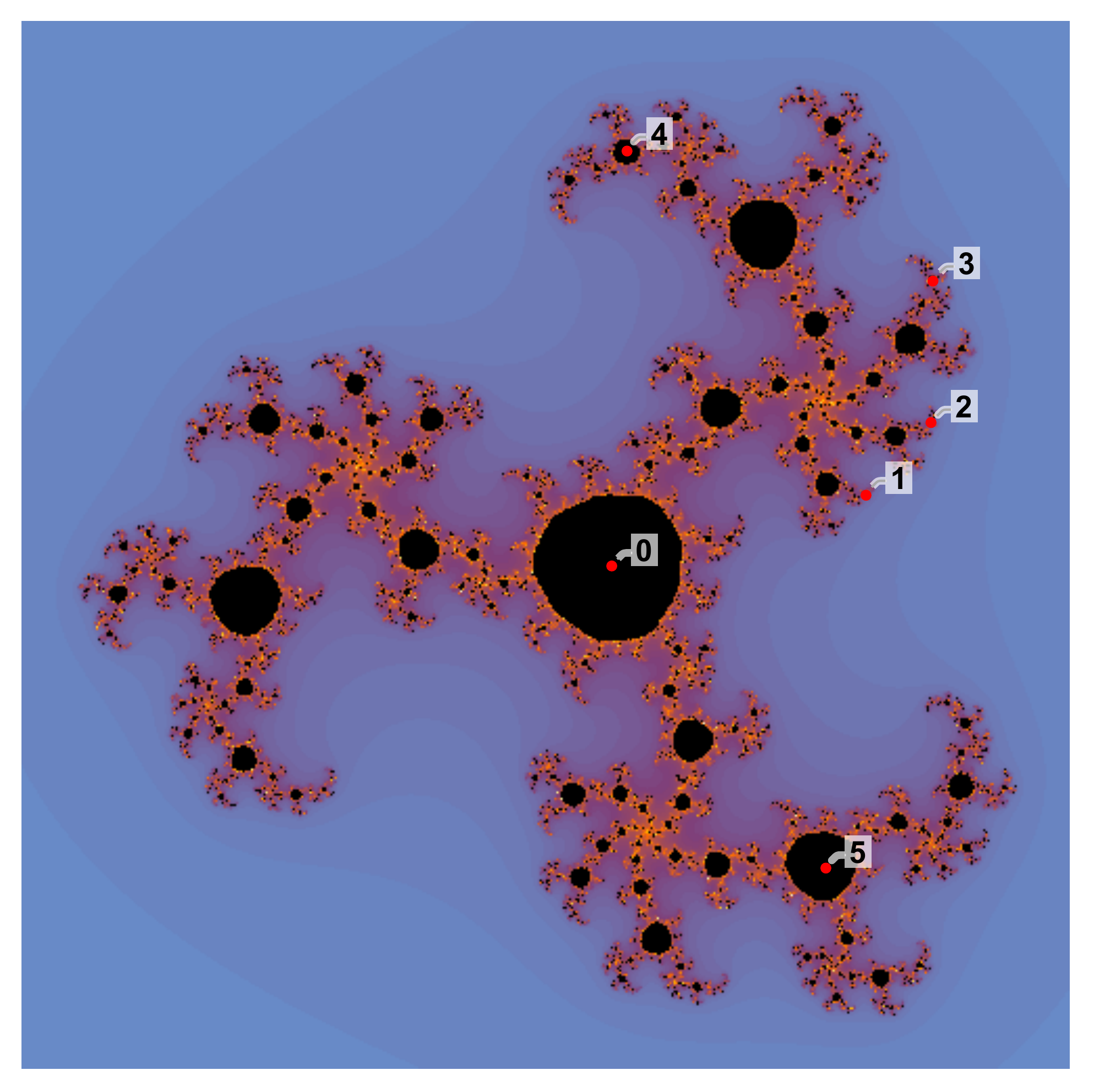}}}$\qquad 
    $\vcenter{\hbox{\labellist\small\hair 2.5pt       
    \pinlabel {$p_{n-1}$} by 0 1 at 0 0
        \pinlabel {$p_0$} by 0 0 at 100 67
         \pinlabel {$p_1$} by 0 0 at 240 100
          \pinlabel {$p_2$} by 0 0 at 285 150
           \pinlabel {$p_3$} by 0 1 at 290 250
        \pinlabel {$p_{n-2}$} by 0 0 at 170 275
              \pinlabel {$e_1$} by 0 0 at 50 25
          \pinlabel {$e_2$} by 0 0 at 167 120
           \pinlabel {$e_3$} by 0 0 at 210 140
             \pinlabel {$e_4$} by 0 0 at 240 157
            \pinlabel {$e_5$} by 0 0 at 250 195
        \pinlabel {$e_n$} by 0 0 at 202 220
        \endlabellist\includegraphics[width=.27\textwidth]{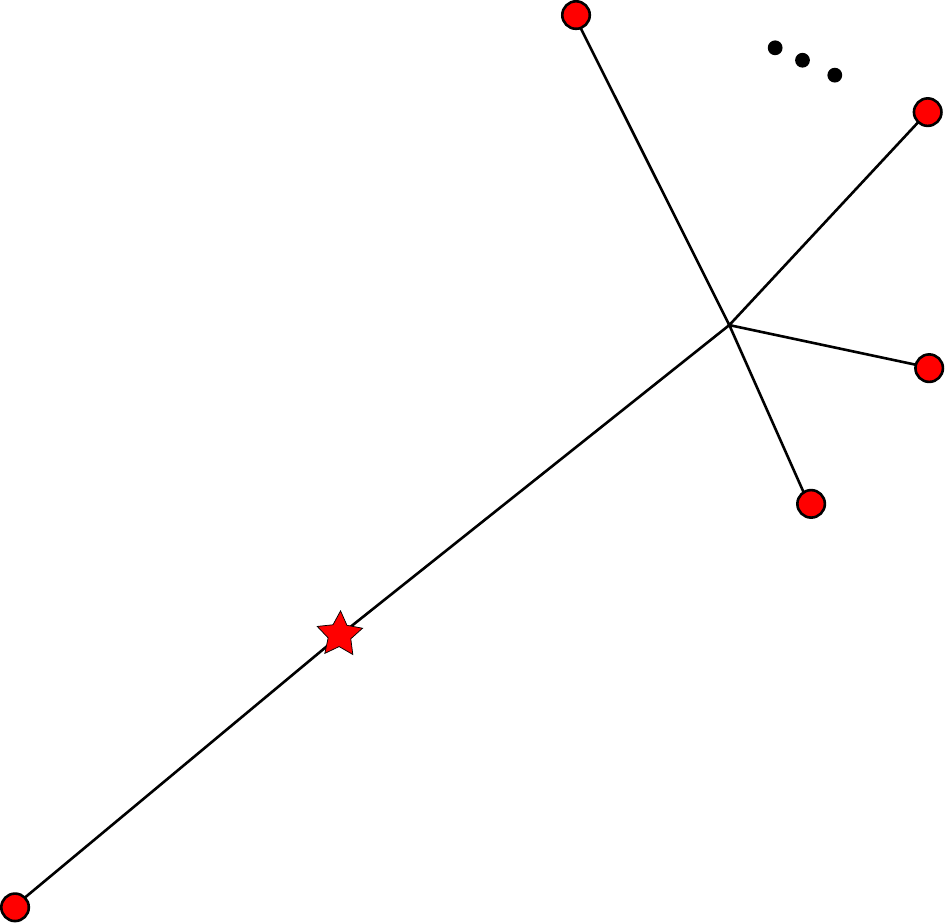}}}$\qquad
      $\vcenter{\hbox{\includegraphics[width=.3\textwidth]{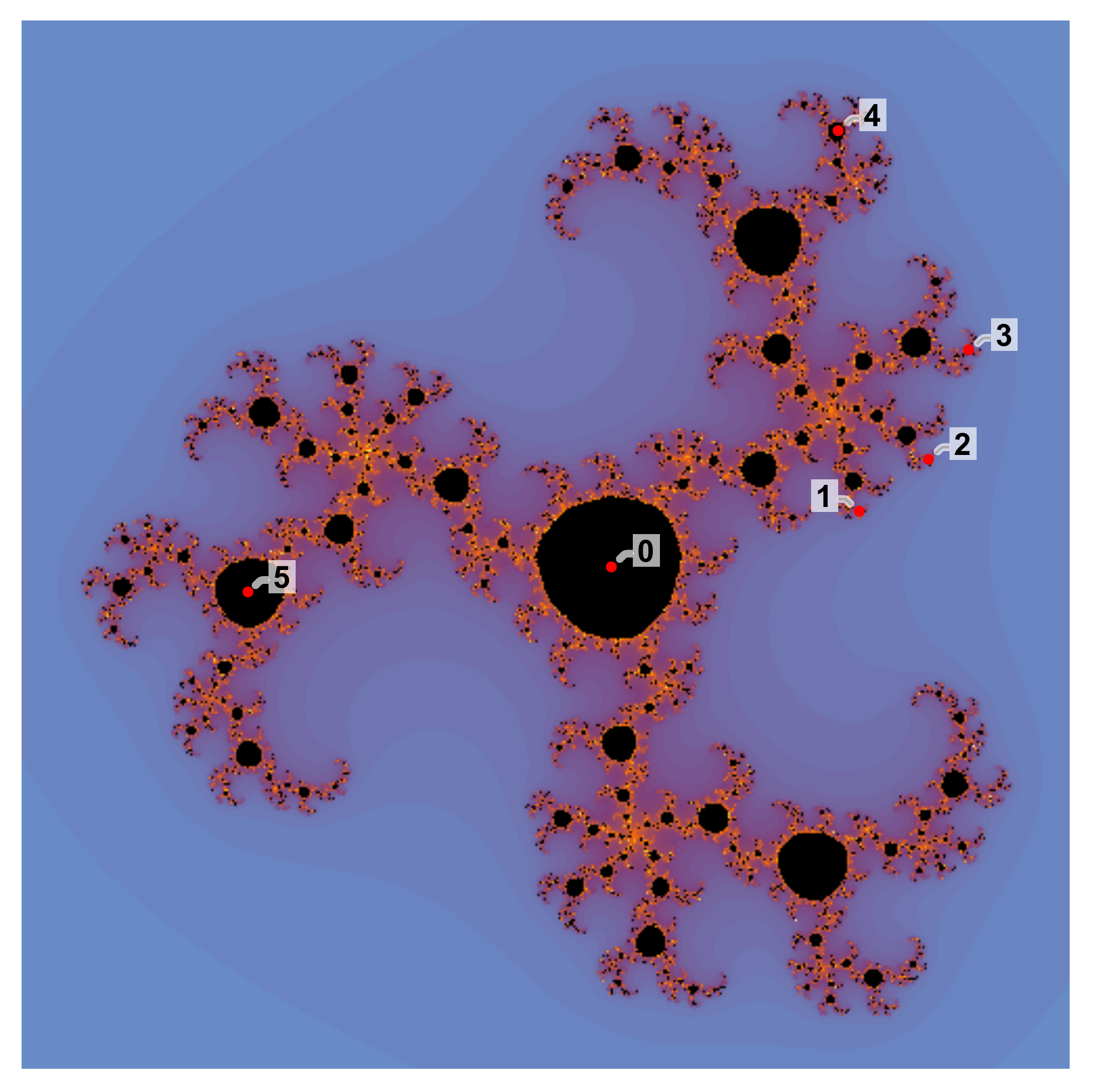}}}$ 
      }
      
      \subcaptionbox{The Julia set for $Y_6$ and the Hubbard tree for $Y_n$\label{fig:Y6}}{   \includegraphics[width=.3\textwidth]{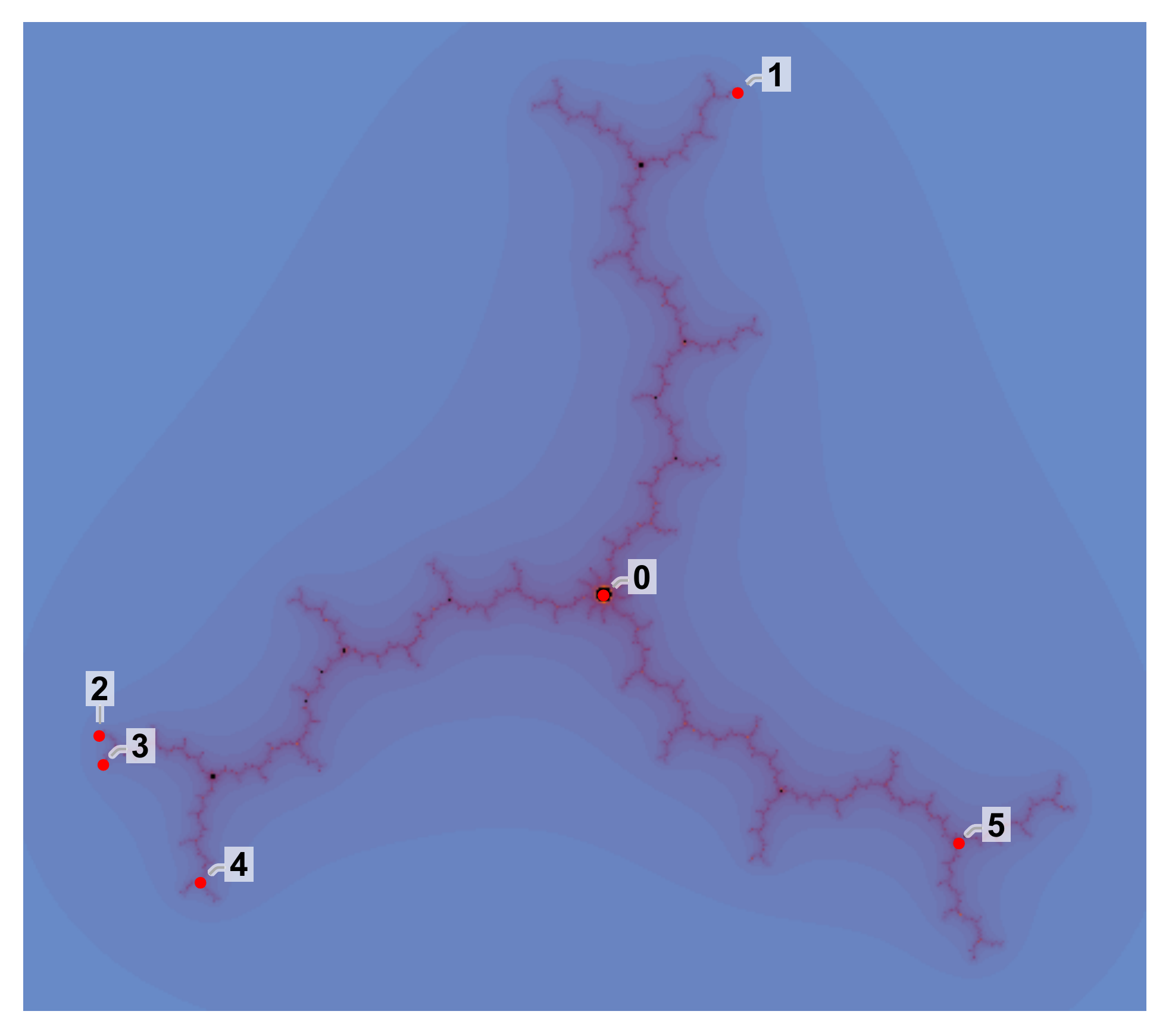}\qquad\qquad
    \labellist\small\hair 2.5pt       
    \pinlabel {$p_{n-1}$} by 0 0 at 320 60
        \pinlabel {$p_0$} by 0 0 at 220 100
         \pinlabel {$p_1$} by 0 0 at 290 240
          \pinlabel {$p_2$} by 1 0 at 0 95
           \pinlabel {$p_{3}$} by 0 0 at 20 45
        \pinlabel {$p_{n-2}$} by 0 0 at 105 0
          \pinlabel {$e_{2n-6}$} by 0 0 at 167 124
           \pinlabel {$e_2$} by 0 0 at 30 104
             \pinlabel {$e_{2n-5}$} by 0 0 at 280 110
            \pinlabel {$e_1$} by 0 0 at 265 195
        \pinlabel {$e_{2n-7}$} by 0 0 at 135 53
              \pinlabel {$e_{2n-8}$} by 0 -.5 at 120 100
              \pinlabel {$e_4$} by 0 0 at 70 104
        \pinlabel {$e_3$} by 0 0 at 55 65
        \endlabellist\includegraphics[width=.35\textwidth]{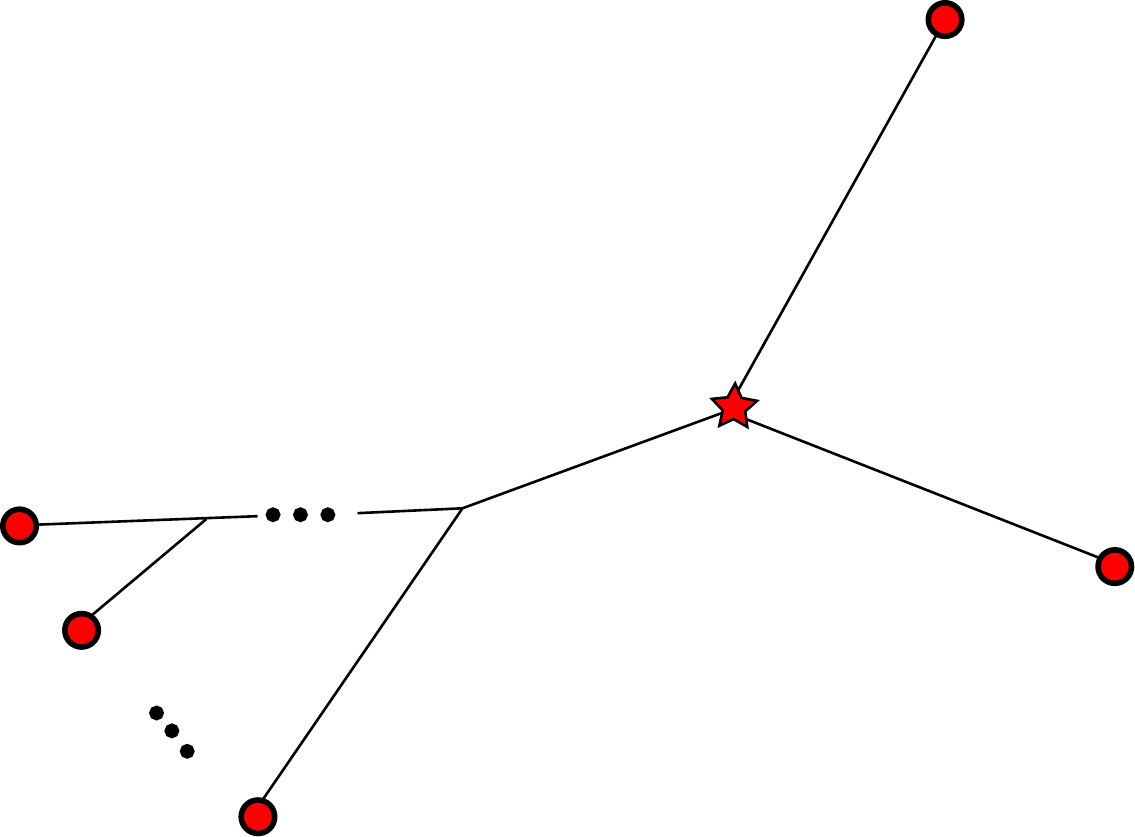}}
    \caption{The Julia sets for $A_6,B_6, K_{6,1},K_{6,2}$ and $Y_6$ and the Hubbard trees for the corresponding families.}
    \label{fig:ngen}
\end{figure}
\p{The family $A_n$} 
After $R_n$, the second family of polynomials is $A_n$. The Hubbard tree for $A_n$ is depicted in Figure~\ref{fig:A6} and the induced map on the edges is:
$$(A_n)_*(e_i) = 
\begin{cases}
e_1e_2\cdots e_{n-1} & i=1 \\
e_{i-1}& 2\leq i\leq n-1
\end{cases}.$$
The complex kneading sequence for the family $A_n$ is $\overline{122\cdots2\ast}$.  There is a homeomorphism between the Hubbard tree for $A_n$ and a subset of the real line such that  $p_1$ is to the left of $p_0$ and $p_2,\cdots, p_{n-1}$ are to the right of $p_0$.  As in the 2-eared cubic rabbit case, let $e_1$ be the edge between $p_0$ and $p_1$, and let $e_2$ be the edge between $p_0$ and $p_{n-1}$.  The angle in $A_n$ measured counterclockwise from $e_1$ to $e_2$ is $2\pi/3$.  This is indicated in the kneading sequence because the second digit is one more than the first, which encodes the fact that $p_2$ is one sector counterclockwise from $p_1$. 
The polynomial $\overline{A}_n$ has the same Hubbard tree as $A_n$, but the invariant angle between $e_1$ and $e_2$ (measured counterclockwise) is $4\pi/3$.  The complex kneading sequence for $\overline{A}_n$ is $\overline{10\cdots 0\ast}$.

\p{The family $B_n$} The third family is $B_n$. The Hubbard tree for $B_n$ is also a path of length $n-1$, depicted in Figure~\ref{fig:B6}.  The induced map on the edges is given by:
$$(B_n)_*(e_i) = 
\begin{cases}
e_3\cdots e_{n-1}\hspace*{2ex} & i=1 \\
e_1e_2& i=2\\
e_1 & i=3 \\
e_2e_3 & i=4 \\
e_{i-1} & 5\leq i\leq n-1
\end{cases}.$$
The cubic complex kneading sequence of $B_n$ is $\overline{12\cdots 21\ast}$ (the intermediate digits are all 2).  There is a homeomorphism between the Hubbard tree for $B_n$ and subset of the real line such that $p_1$ and $p_{n-1}$ are to the left of $p_0$ and all the other post-critical points are to the right of $p_0$.  We name the edges $e_1,\cdots, e_{n-1}$ from left to right in this embedding.  Then $e_2$ and $e_3$ are the edges that meet at the critical point.  The (counterclockwise) angle between $e_2$ and $e_3$ has measure $2\pi/3$.  This is indicated in the kneading sequence because the $n-1$st digit is one less than the $n-2$nd, which encodes the fact that $p_{n-2}$ is one sector counterclockwise from $p_{n-1}$ (and $p_{n-1}$ and $p_{n-2}$ are the endpoints of the edges $e_1$ and $e_2$ respectively).  The polynomial $\overline{B}_n$ has the same Hubbard tree as $B_n$ but the invariant (counterclockwise) angle between $e_2$ and $e_3$ has measure $4\pi/3$.  The complex kneading sequence for $\overline{B}_n$ is $\overline{10\cdots 0 1\ast}$.

\p{The families $K_{n,1}$ and $K_{n,2}$} The fourth and fifth families $K_{n,1}$ and $K_{n,2}$ have the same Hubbard tree and differ only by their invariant angle assignments. Unlike in the families $A_n$ and $B_n$, this difference cannot be accounted for with complex conjugation. (As a consequence, we are in effect defining four families: $K_{n,1}$, $\overline{K}_{n,1}$, $K_{n,2}$, and $\overline{K}_{n,2}$.) The Hubbard tree for $K_{n,1}$ and $K_{n,2}$ is depicted in Figure \ref{fig:K6}; we label $e_1$ as the edge between the $p_0$ and $p_{n-1}$.  
For $j\in\{1,2\}$, the induced map on the edges of the Hubbard tree is given by:
\begin{align*}(K_{n,j})_*(e_i) &= 
\begin{cases}
e_2e_3\hspace*{7.9ex}  & i=1 \\
e_{i+1}& 2\leq i\leq n-1\\
e_1e_2 & i=n
\end{cases}.\end{align*}

 The polynomial $K_{n,1}$ has cubic complex kneading sequence $\overline{1\cdots 1 0\ast}$.  The (counterclockwise) angle between $e_1$ and $e_2$ has measure $2\pi/3$.  This is indicated in the kneading sequence because the first digit is one more than the $(n-1)$st digit, which encodes the fact that $p_1$ is one sector counterclockwise from $p_{n-1}$ (and $p_{n-1}$ and $p_{1}$ are the endpoints of $e_1$ and $e_2$, respectively).  
 
 The polynomial $K_{n,2}$ has cubic complex kneading sequence $\overline{1\cdots 1 2\ast}$.  The (counterclockwise) angle between $e_1$ and $e_2$ has measure $4\pi/3$.  This is indicated in the kneading sequence because the first digit is two more than the $(n-1)$st digit, which encodes the fact that $p_1$ is two sectors counterclockwise from $p_{n-1}$ (and $p_{n-1}$ and $p_{1}$ are the endpoints of $e_1$ and $e_2$, respectively).
 
 Both polynomials $K_{n,1}$ and $K_{n,2}$ map the post-critical set counterclockwise relative to the critical point.
 
 While the polynomials $\overline{K}_{n,1}$ and $\overline{K}_{n,2}$ do not arise in the answer to the cubic twisted many-eared rabbit problem, we observe that both polynomials rotate the post-critical set clockwise with respect to the critical point.  The kneading sequence for $\overline{K}_{n,1}$ is $\overline{1\cdots 1 2\ast}$ and the kneading sequence for $\overline{K}_{n,2}$ is $\overline{1\cdots 1 0\ast}$.

\p{The family $Y_n$} Unlike the polynomials we have seen so far, the sixth and final family has a trivalent critical point.  This is not possible for quadratic polynomials or polynomials with three post-critical points, so it does not have a direct analogue to any family of polynomials that appear in previously considered twisting problems.  We call this family $Y_n$ to reflect the trivalent critical point. 
The Hubbard tree for $Y_6$ is depicted in Figure \ref{fig:Y6}. 
The induced map on the edges of the Hubbard tree is given by: 
$$(Y_n)_*(e_i) = 
\begin{cases}
e_2e_4\cdots e_{2n-6}e_1 & i=1 \\
e_3e_4 & i=2\\
e_{i+2} & 3 \leq i \leq 2n-7\\
e_1 & i = 2n-6 \\
e_1 & i = 2n-5 \\
\end{cases}.$$

The polynomial $Y_n$ has cubic complex kneading sequence $\overline{120\cdots 0\ast}$.  The critical point is valence 3 and the angle between each pair of edges that are adjacent in the cyclic order around the critical point is $2\pi/3$.  The map $Y_n$ permutes the vertices of $Y_n$ counterclockwise with respect to the critical point.  This is coarsely seen in the kneading sequence because the second digit is one more than the first and the $(n-1)$st digit is one more than the $(n-2)$nd, which encodes the fact that $p_2$ (and $p_3,\cdots,p_{n-2}$) is one sector counterclockwise from $p_1$ and $p_{n-1}$ is one sector counterclockwise from $p_{n-2}$. 

On the other hand, the polynomial $\overline{Y}_n$ has complex kneading sequence $\overline{10\cdots 02\ast}$.  The map $\overline{Y}_n$ permutes the vertices of $\overline{Y}_n$ counterclockwise with respect to the critical point.

\subsection{Reduction formulas}As in Section \ref{sec:cubic} (and Bartholdi--Nekrashevych \cite{BaNe}), the first step in solving the many-eared twisted cubic rabbit problem is to compute reduction formulas.  

In what follows, we assume $n \geq 4$.  The reduction formulas are as follows.

\label{n_reduction}
Let $m \in \Z$.  Then 
\[
D_x^{m} R_n\simeq 
\begin{cases}
D_x^k R_n    & m=9k\\
D_x R_n      & m=9k+1\\
D_x^2 R_n    & m=9k+2\\
D_y R_n      & m=9k+3\\
D_y D_x R_n    & m=9k+4\\
D_y^{-1}D_x^{-1} R_n      & m=9k+5\\
D_y^2 R_n    & m=9k+6\\
D_x^{-2} R_n & m=9k+7.\\
D_x^{-1} R_n    & m=9k+8\\
\end{cases}.
\]

The reduction formulas for the twisted many-eared cubic rabbit problem are similar to the reduction formulas for the twisted cubic rabbit problem from the previous section. Some of the calculations are in fact identical. The only cases that are different are when $m$ is $9k+5, 9k+6,$ or $9k+8$.  The difference in these cases is that when $n=3$, we have the lantern relation $D_z=D_y^{-1}D_x^{-1}$.  The same relation does not hold when $n\geq4$, which affects the calculation in two different ways.  The first difference affects the base cases when $m=9k+6$. If $n\geq 4$ and $m=9k+6$, then $D_x^mR_n$ is equivalent to $D_y^{-1}D_x^{-1}R_n$.  When $n=3$ and $m=9k+6$, the same sequence of steps hold, but the lantern relation gives us that $D_y^{-1}D_x^{-1}R_n=D_zR_n$, which is equivalent to $D_xR_n$, hence yielding a different base case.
The second difference arises when $m=9k+5$ and $m=9k+8$.  The calculations for the reduction formulas when $m=9k+5$ involve lifting the curve $D_y^{-1}D_x^{-1}(z)$ and the calculations for the reduction formulas when $m=9k+8$ involve lifting the curve $D_xD_y(z)$.  When $n\geq 4$, the curves $D_y^{-1}D_x^{-1}(z)$ and $D_xD_y(z)$ both lift trivially, which allows these cases to reduce to a base case.  On the other hand, when $n=3$, the curves $D_y^{-1}D_x^{-1}(z)$ and $D_xD_y(z)$ are both equal to $z$ (by the lantern relation), which results in the reduction formula $D_x^k$.

Before beginning our justifications of the reduction formulas, we observe several preliminary facts. For all $i \neq 1$ or $2$, we have $D_{c_i} \leadsto D_{c_{i-1}}$; in particular we have $D_{c_3}\leadsto D_{c_2}=D_x$.  As in the case of $R_3$, for all $n\geq 3$, the straight line arc $b$ from the critical value $p_1$ to infinity is a special branch cut for $R_n$.  We again have that $D_{D_y^k(z)} \leadsto \textrm{id}$ by Lemma \ref{lem:triviality2.0}. Unlike the case where $n=3$, when $n\geq 4$ the curves $x$ and $z$ are disjoint and so the Dehn twists $D_z$ and $D_x$ commute; therefore $D_{D_x^{k}(z)}=D_z \leadsto D_{c_{n-1}}$. On the other hand, $D_{D_x^k(c_3)} \leadsto \textrm{id}$ when $k\not\equiv 0\mod 3$ by Lemma \ref{lem:triviality2.0}. 

\bigskip

Here are the calculations justifying the reduction formulas:

\bigskip

\noindent \emph{Case 1: $m=9k$.} We have
$D_x^{9k}=(D_x^3)^{3k} \ 
\leadsto \ 
D_y^{3k} = (D_y^3)^k \ 
\leadsto \ 
D_z^k \ 
\leadsto \ 
\dots \
\leadsto
D_x^k.$

\bigskip

 \noindent\emph{Case 2: $m=9k+1$.}  As above, we use the fact that $D^3_{D_x^{-1}(y)} \leadsto D_z$. We have:
 
$D_x^{9k+1} \ 
\stackrel{D_x}{\leadsto} \ 
D_y^{3k}D_x \ 
\stackrel{D_x}{\leadsto} \ 
D_z^k D_x \ 
\stackrel{D_x}{\leadsto} \ 
D_{c_{n-1}}^kD_x \
\stackrel{D_x}{\leadsto}
\dots
\stackrel{D_x}{\leadsto}
D_{c_3}^kD_x
\stackrel{D_x}{\leadsto}
D_x.
$

\bigskip

\noindent\emph{Case 3: $m=9k+2$.} As above, we use the fact that the preimage of the curve $D_x^{-2}(y)$ has a single component, which is isotopic to $z$.  We also use that $D_x$ commutes with $D_{c_i}$ for $i \neq 1$ or $3$.  We have:

$D_x^{9k+2} \ 
\stackrel{D_x^{2}}\leadsto \  D_y^{3k}D_x^2 \ 
\stackrel{D_x^2}{\leadsto} \ 
D_z^k D_x^2 \ 
\stackrel{D_x^2}{\leadsto} \ 
\dots
\stackrel{D_x^2}{\leadsto} \ 
D_{c_3}^k D_x^2 \ 
\stackrel{D_x^2}{\leadsto} \ 
D_x^2.$

\bigskip

\noindent\emph{Case 4: $m=9k+3$.}  We have
$D_x^{9k+3} \ 
{\leadsto} \ 
D_y^{3k+1} \ 
\stackrel{D_y}\leadsto \ 
D_z^k D_y \ 
\stackrel{D_y}\leadsto \ 
D_y.$

\bigskip

\noindent\emph{Case 5: $m=9k+4$.} We have 
$D_x^{9k+4} \ 
\stackrel{D_x}{\leadsto} \
D_y^{3k+1}D_x \ 
\stackrel{D_yD_x}{\leadsto} \
D_z^{k}D_yD_x
\stackrel{D_yD_x}{\leadsto} \
D_yD_x.
$

\bigskip

\noindent\emph{Case 6: $m=9k+5$.} As above, we use two additional facts.  First, the preimage of $D_x(y)$ is a single component that is homotopic to $D_y(z)$.  Second, the curve $D_xD_y^2(z)$ has algebraic intersection 1 with $b$, therefore the preimage of $D_xD_y^2(z)$ is trivial by Lemma~\ref{lem:triviality2.0}.  We have
\begin{align*}
D_x^{9k+5} \ 
\stackrel{D_x^{-1}}{\leadsto} \ 
D_y^{3k+2}D_x^{-1} \  \stackrel{D_y^{-1}D_x^{-1}}{\leadsto} \ 
\psi(D_x D_y^{3k+3} D_x^{-1})D_y^{-1}D_x^{-1}=
(D_y D_z^{k+1} D_y^{-1}) D_y^{-1}D_x^{-1}
\stackrel{D_y^{-1}D_x^{-1}}\leadsto
\end{align*}
\begin{align*}
\psi(D_xD_y^2D_z^{k+1}D_y^{-2}D_x^{-1})D_y^{-1}D_x^{-1}=D_y^{-1}D_x^{-1}.
\end{align*}

\bigskip

\noindent\emph{Case 7: $m=9k+6$.} We have $
D_x^{9k+6} \
{\leadsto} \ 
D_y^{3k+2} \ 
\stackrel{D_y^2}\leadsto \ 
D_z^k D_y^2 \ 
\stackrel{D_y^2}\leadsto \ 
D_y^2.
$

\bigskip

\noindent\emph{Case 8: $m=9k+7$.} In this case we use the fact that the preimage of $D_x^2(y)$ consists of single connected component, which is homotopic to $D_{x}^{-1}(z)$: $$
D_x^{9k+7} \ 
\stackrel{D_x^{-2}}{\leadsto} \ 
D_y^{3k+3}D_x^{-2} \ 
\stackrel{D_x^{-2}}\leadsto \ 
D_yD_z^{k+1}D_y^{-1}D_x^{-2} \ 
\stackrel{D_x^{-2}}\leadsto \ D_x^{-2}.
$$

\bigskip

\noindent\emph{Case 9: $m=9k+8$.} In this case we use the fact that the preimage of the curve $D_x(y)$ consists of single connected component, which is homotopic to $D_y(z)$. We also use the fact that $D_yD_zD_y^{-1}$ lifts to the identity by Lemma~\ref{lem:triviality2.0}.  We have: $$
D_x^{9k+8} \ 
\stackrel{D_x^{-1}}{\leadsto} \
D_y^{3k+3} D_x^{-1} \ 
\stackrel{D_x^{-1}}\leadsto \ 
D_yD_z^{k+1}D_y^{-1} D_x^{-1} \ 
\stackrel{D_x^{-1}}\leadsto \ 
D_x^{-1}.
$$

\subsection{Base cases}To complete our solution to the twisted many-eared rabbit problem, it remains to determine the base cases.    

\begin{figure}
    \centering
   $\underset{\textstyle\text{(a) $D_xR_n$ and $D_x^2R_n$}}{\labellist\small\hair 2.5pt       
    \pinlabel {$p_0$} by 0 0 at -2 5
        \pinlabel {$p_1$} by 0 0 at 125 15
         \pinlabel {$p_2$} by 0 0 at 190 125
          \pinlabel {$p_{n-1}$} by 0 0 at 100 190
           \pinlabel {$e_1$} by 0 0 at 60 30
           \pinlabel {$e_2$} by 0 0 at 45 100
            \pinlabel {$e_{n-1}$} by 0 1 at 140 130
        \endlabellist\includegraphics[width=.2\textwidth]{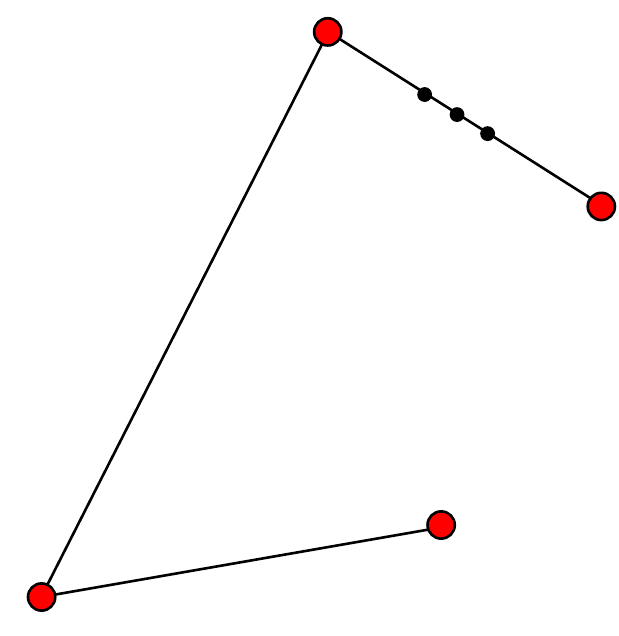}}$\quad
   $\underset{\textstyle\text{(b) $D_yR_n$ and $D_y^2R_n$}}{\labellist\small\hair 2.5pt       
    \pinlabel {$p_0$} by 0 0 at -9 2
        \pinlabel {$p_1$} by 0 0 at 125 10
         \pinlabel {$p_2$} by 0 0 at 185 120
          \pinlabel {$p_{n-1}$} by 0 0 at 90 190
           \pinlabel {$e_1$} by 0 0 at 50 50
           \pinlabel {$e_2$} by 0 0 at 20 100
            \pinlabel {$e_3$} by 0 0 at 127 50
            \pinlabel {$e_4$} by 0 0 at 145 96
        \endlabellist\includegraphics[width=.2\textwidth]{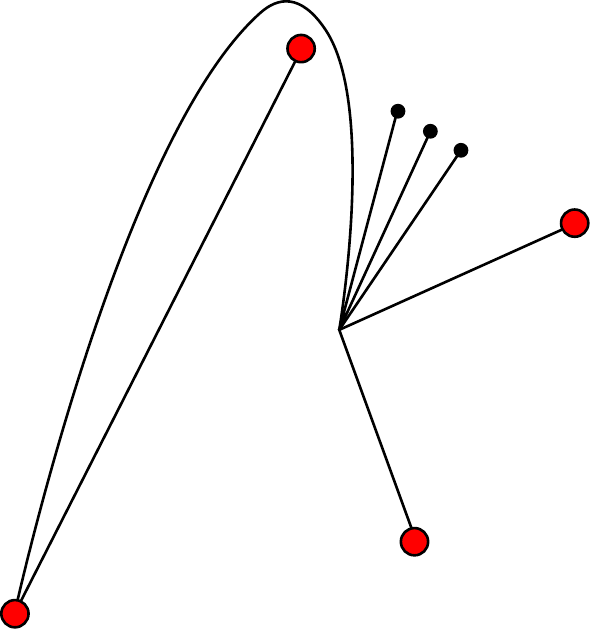}}$\quad
   $\underset{\textstyle\text{(c) $D_yD_xR_n$}}{\labellist\small\hair 2.5pt       
    \pinlabel {$p_0$} by 0 0 at -9 2
        \pinlabel {$p_1$} by 0 0 at 125 10
         \pinlabel {$p_2$} by 0 0 at 185 115
          \pinlabel {$p_{n-1}$} by 0 0 at 90 190
          \pinlabel {$e_1$} by 0 0 at 100 70
           \pinlabel {$e_2$} by 0 0 at 47 50
           \pinlabel {$e_3$} by 0 0 at 10 100
            \pinlabel {$e_{n-1}$} by 0 0 at 158 135
        \endlabellist\includegraphics[width=.2\textwidth]{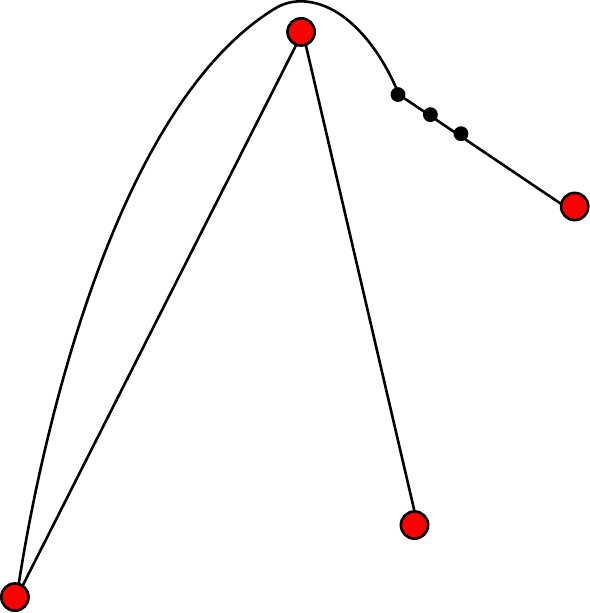}}$\quad
   $\underset{\textstyle\text{(d) $D_y^{-1}D_x^{-1}R_n$}}{\labellist\small\hair 2.5pt   
    \pinlabel {$p_0$} by 0 0 at -5 10
        \pinlabel {$p_1$} by 0 0 at 105 45
         \pinlabel {$p_2$} by 0 0 at 185 135
          \pinlabel {$p_{n-1}$} by 0 0 at 90 210
          \pinlabel {$e_1$} by 0 0 at 90 90
           \pinlabel {$e_2$} by 0 0 at 170 150
           \pinlabel {$e_{2n-3}$} by 0 0 at 58 36
            \pinlabel {$e_{2n-4}$} by 0 0 at 200 10
        \endlabellist\includegraphics[width=.2\textwidth]{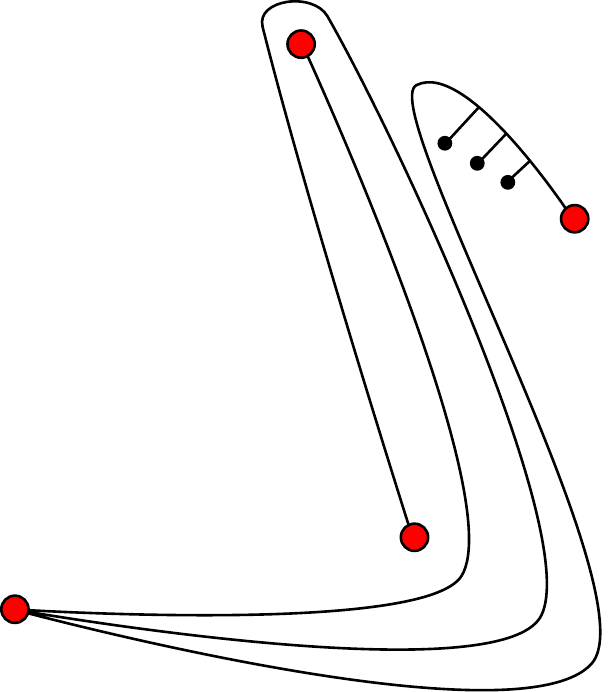}}$\quad
   $\underset{\textstyle\text{(e) $D_x^{-2}$}}{\labellist\small\hair 2.5pt   
    \pinlabel {$p_0$} by 0 0 at -9 2
        \pinlabel {$p_1$} by 0 0 at 105 45
         \pinlabel {$p_2$} by 0 0 at 185 135
          \pinlabel {$p_{n-1}$} by 0 0 at 60 175
          \pinlabel {$e_{2n-3}$} by 0 0 at 90 90
           \pinlabel {$e_2$} by 0 0 at 175 100
           \pinlabel {$e_{1}$} by 0 0 at 58 36
            \pinlabel {$e_{2n-4}$} by 0 0 at 160 0
        \endlabellist\includegraphics[width=.2\textwidth]{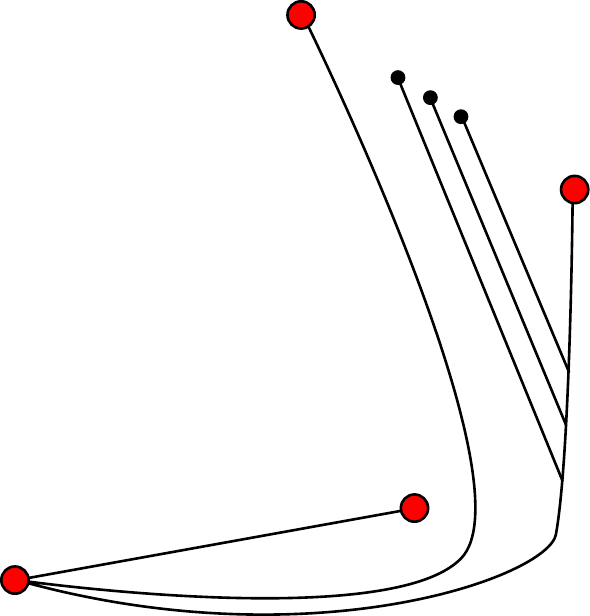}}$\qquad
   $\underset{\textstyle\text{(f) $D_x^{-1}R_n$}}{\labellist\small\hair 2.5pt       
    \pinlabel {$p_0$} by 0 0 at -9 2
        \pinlabel {$p_1$} by 0 0 at 105 45
         \pinlabel {$p_2$} by 0 0 at 185 125
         \pinlabel {$p_{n-1}$} by 0 0 at 60 178
          \pinlabel {$e_1$} by 0 0 at 95 90
           \pinlabel {$e_2$} by 0 0 at 60 25
           \pinlabel {$e_3$} by 0 0 at 160 80
           \pinlabel {$e_{n-1}$} by 0 0 at 160 145
        \endlabellist\includegraphics[width=.2\textwidth]{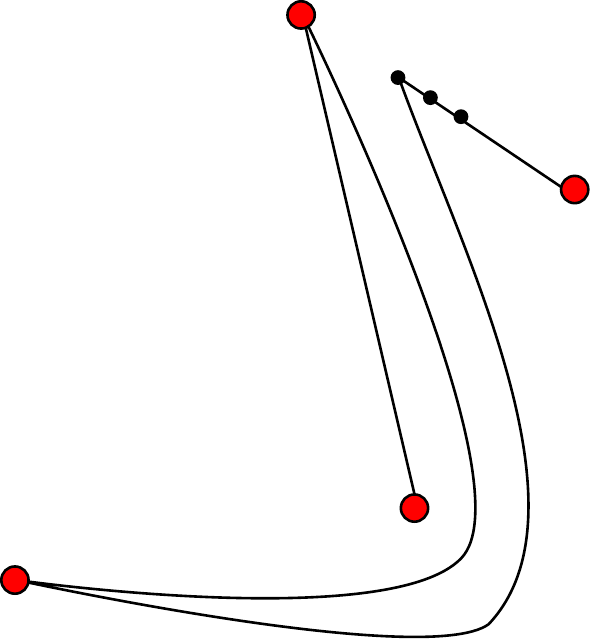}}$
    \caption{The topological Hubbard trees for the indicated maps.}
    \label{fig:invariant_trees}
\end{figure}
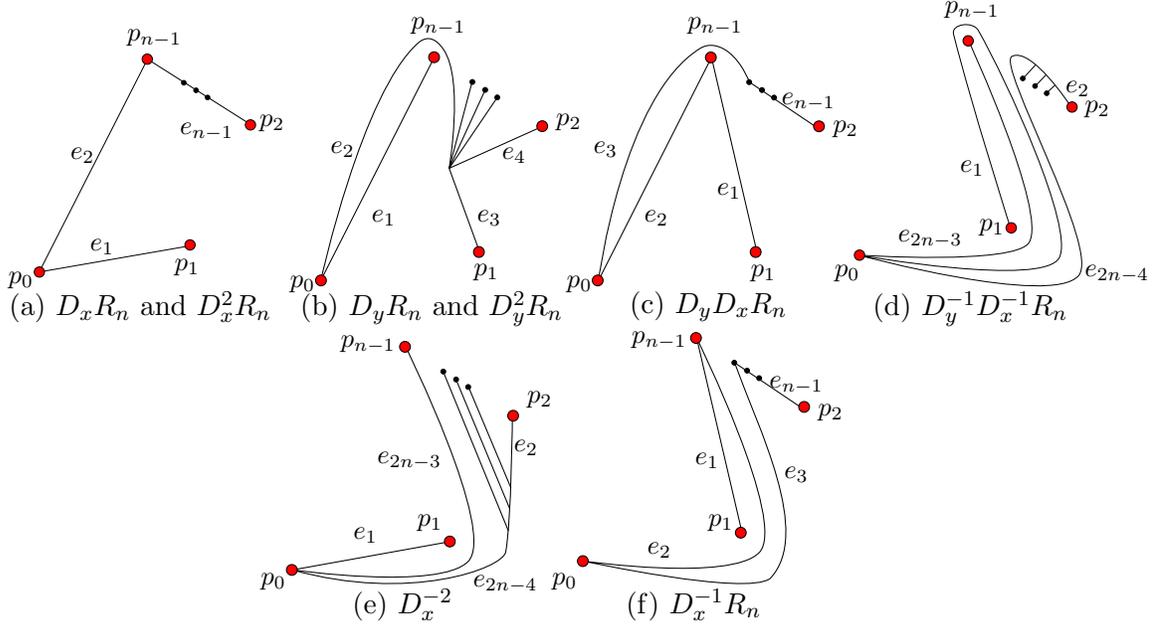

\begin{proof}[Proof of Theorem~\ref{thm:many}]

It follows from the reduction formulas that when $m \neq 0$ the map $D_x^m R_n$ is equivalent to one of the base cases, depending only on $\dig$, the first non-zero digit of the 9-adic expansion of $m$. 

It remains to determine the polynomials to which the base cases are equivalent. We do this by applying the Alexander method \cite[Proposition 3.1]{BLMW}.  For each base case, let $g$ be the twisted rabbit under consideration and let $f$ be the polynomial to which we claim it is equivalent. We check that $f$ and $g$ are equivalent by checking that the preimage of the Hubbard tree for $f$ is homeomorphic to the preimage of the topological Hubbard tree for $g$, that the invariant angle assignments of the trees are equal, and that the dynamical maps agree on the edges of the preimage trees.  Topological Hubbard trees for each of the base cases are depicted in Figure~\ref{fig:invariant_trees}.  The maps $D_xR_n$ and $D_x^2R_n$ have the same topological Hubbard tree, but the (counterclockwise) angle between $e_1$ and $e_2$ that is invariant under $D_xR_n$ is $2\pi/3$ and the angle that is invariant under $D_x^2R_n$ is $4\pi/3$. Thus $D_xR_n$ is equivalent to $A_n$ and $D_x^2R_n$ is equivalent to $\overline{A}_n$.  The maps $D_yR_n$ and $D_y^2R_n$ have the same topological Hubbard tree, but the (counterclockwise) angle between $e_1$ and $e_2$ that is invariant under $D_yR_n$ is $\frac{2\pi}{3}$ and the invariant angle under $D_y^2R_n$ is $\frac{4\pi}{3}$. Indeed, the bottom row of Figure~\ref{fig:tree_lift} demonstrates the process of lifting the topological Hubbard for $D_yR_n$ through the map $D_yR_n$.  Moreover, it also shows that the edges of the lift of the topological Hubbard tree that are homotopic (relative to endpoints) to $e_1$ and $e_2$ have angle measure $\frac{2\pi}{3}$.  Thus $D_yR_n$ is equivalent to $K_{n,1}$ and a similar calculation shows that $D_y^2R_n$ is equivalent to $K_{n,2}$.  The topological Hubbard tree for $D_yD_xR_n$ is a path of length $n-1$ with the same edge map as $B_n$, and indeed the angle between $e_2$ and $e_3$ (measured counterclockwise) is $\frac{2\pi}{3}$.  Therefore $D_yD_xR_n$ is equivalent to $B_n$.  On the other hand, the topological Hubbard tree for $D_x^{-1}R_n$ has the same edge map as $B_n$, but the (counterclockwise) angle between $e_2$ and $e_3$ is $\frac{4\pi}{3}$.  Therefore $D_x^{-1}R_n$ is equivalent to $\overline{B}_n$.  The topological Hubbard tree for $D_y^{-1}D_x^{-1}R_n$ is trivalent at the critical point and the edges $\{e_1,e_{2n-3},e_{2n-4}\}$ are cyclically permuted counterclockwise, which agrees with the edge map for $\overline{Y}_n$.  Similarly, the topological Hubbard tree for $D_x^{-2}R_n$ is also trivalent at the critical point.  However, $D_x^{-2}R_n$ permutes the edges  $\{e_1,e_{2n-3},e_{2n-4}\}$ cyclically permutes the edges clockwise, which agrees with the edge map for $Y_n$.
\end{proof}

As mentioned in Section 5 of \cite{BLMW},
the tree lifting algorithm is used to find the topological Hubbard tree for each base case. Once found, however, the earlier steps of the algorithm are not needed in order to verify that the output tree is invariant under lifting. We illustrate both the algorithm and the invariance check for the map $D_yR_n$ from case $9k+3$ in Figure~\ref{fig:tree_lift}. Using the Hubbard tree for $R_n$ as the input tree, the tree lifting algorithm requires only a single step in order to find the topological Hubbard tree for $D_yR_n$.

\p{More generalizations} Notice that because each of $c_3,c_4,\dots,c_{n-1}$ and $c_0=z$ lifts to $x$ under iteration of the lifting map for $R_n$, we have that $D_{c_i}^m \simeq D_x^m$ for all $i\neq 1$ and for all $m \in \Z$. Thus our solution to the twisted rabbit problems of the form $D_x^mR_n$ also immediately give solutions to all twisted rabbit problems of the form $D_{c_i}^mR_n$ with $i \neq 1$.

\begin{figure}[h]
    \centering
    $\vcenter{\hbox{\includegraphics[scale=.5]{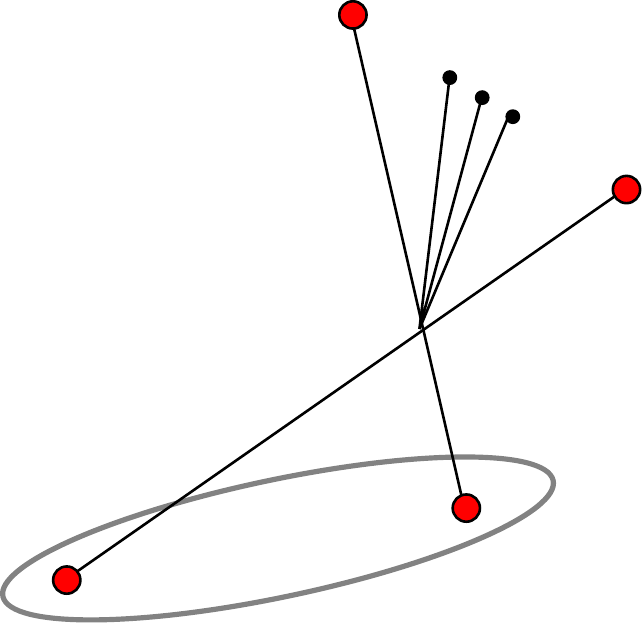}}}\quad \stackrel{D_y^{-1}}{\longrightarrow}\quad\vcenter{\hbox{\includegraphics[scale=.5]{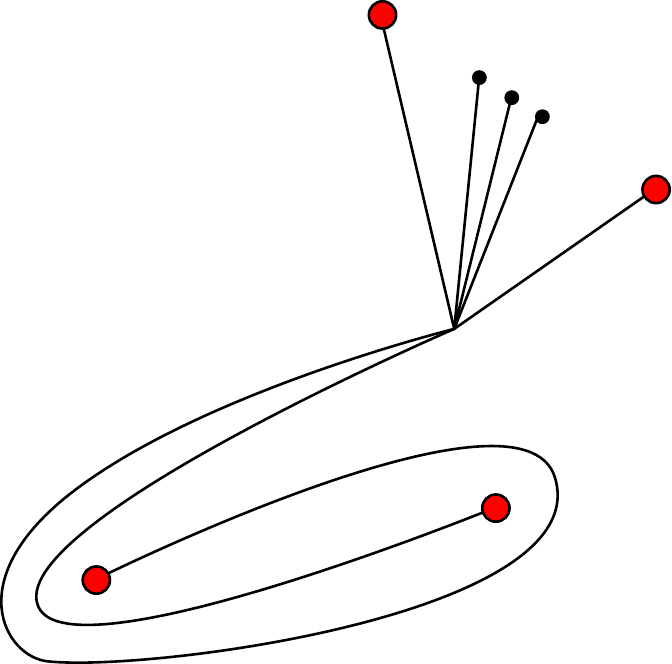}}}\quad\stackrel{R_n^{-1}}{\longrightarrow}\quad\vcenter{\hbox{\includegraphics[scale=.6]{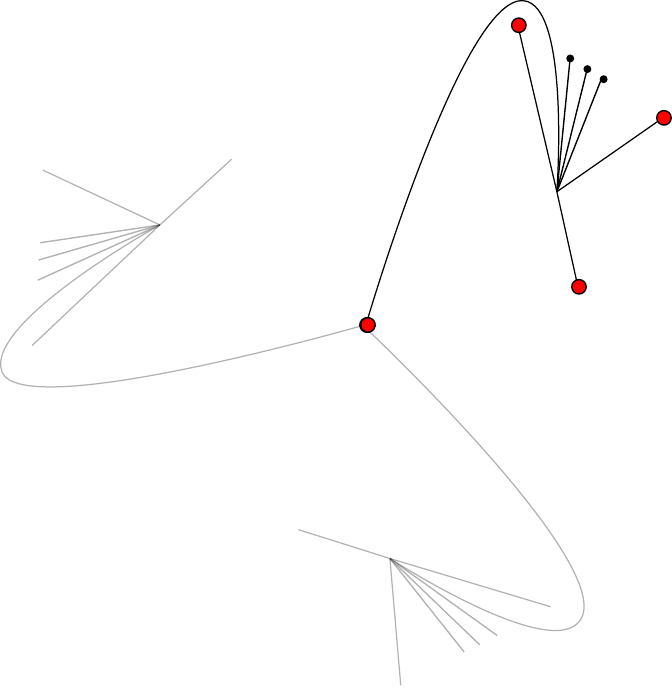}}}$
    
    $\vcenter{\hbox{\includegraphics[scale=.5]{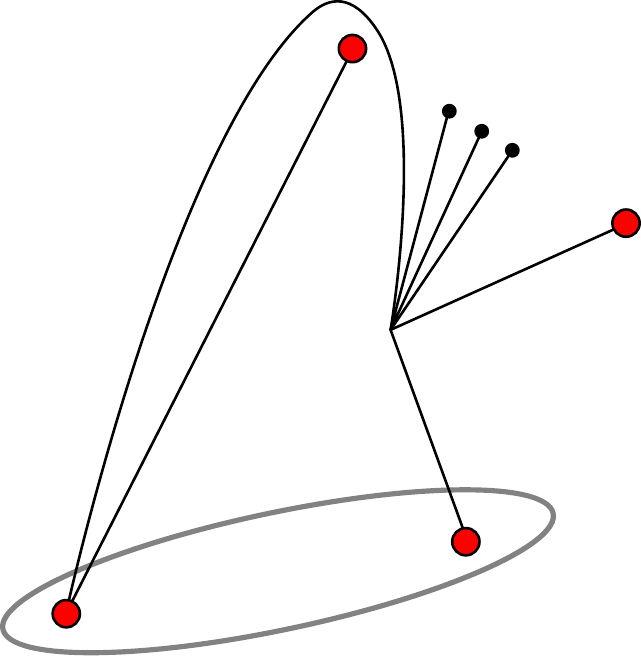}}}\quad \stackrel{D_y^{-1}}{\longrightarrow}\quad\vcenter{\hbox{\includegraphics[scale=.5]{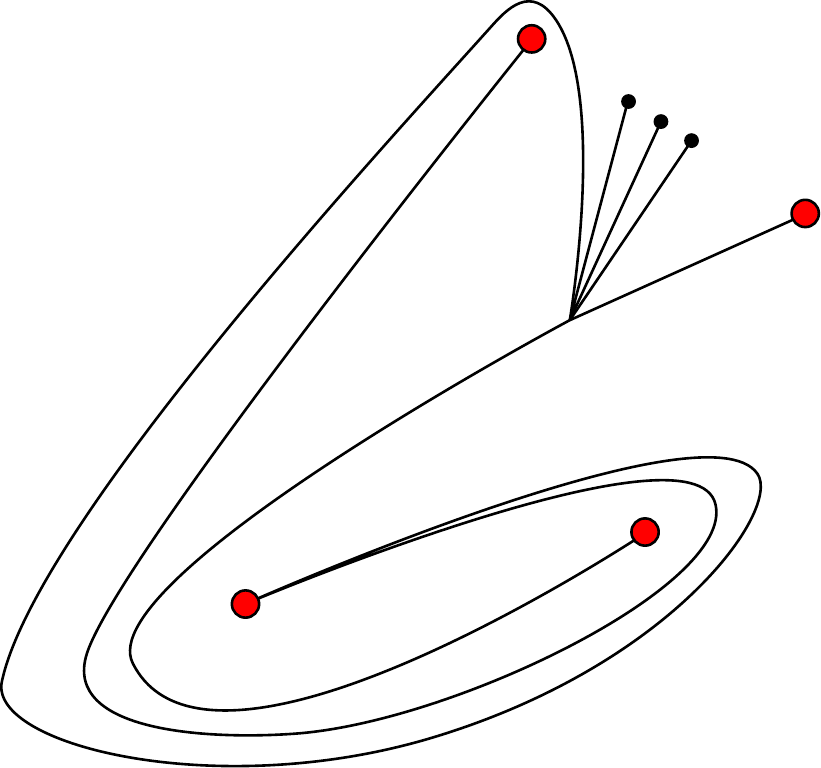}}}\quad\stackrel{R_n^{-1}}{\longrightarrow}\quad\vcenter{\hbox{\labellist\small\hair 2.5pt       
    \pinlabel {$p_0$} by 0 0 at 110 92
          \pinlabel {$e_1$} by 0 0 at 152 80
           \pinlabel {$e_2$} by 0 0 at 140 140
        \endlabellist\includegraphics[scale=.6]{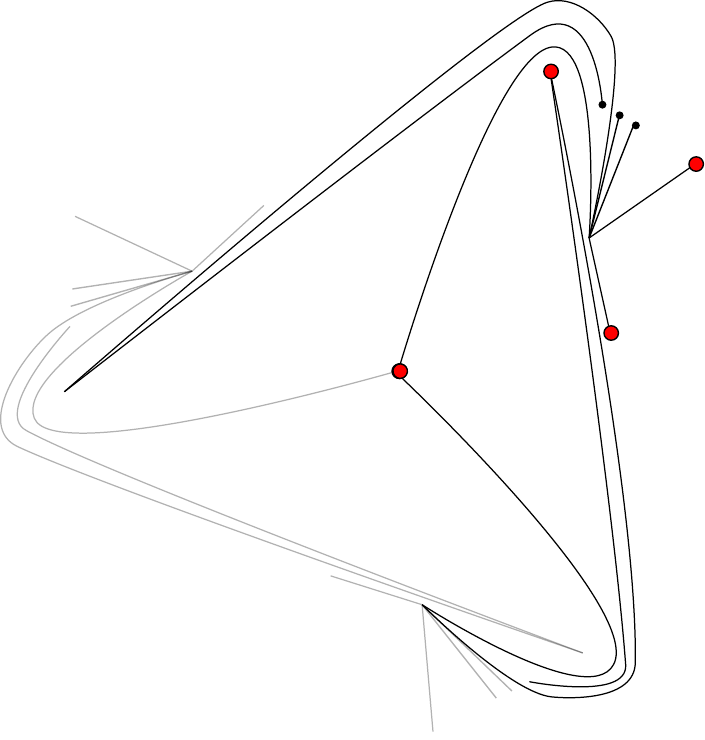}}}$
    \caption{The top row shows the lift of the Hubbard tree for $R_n$ under $D_yR_n$.  The second row shows that the resulting tree is invariant under $D_yR_n$.  In the figures on the right, the dark edges are those that are in the hull of the post-critical set.}
    \label{fig:tree_lift}
\end{figure}

\bibliographystyle{plain}
\bibliography{refs.bib}

\end{document}